\documentclass[a4paper, 11pt]{article}
\usepackage[utf8]{inputenc}
\usepackage{amssymb}
\usepackage{amsmath}
\usepackage{amsthm}
\usepackage{bbm}
\usepackage{dsfont}
\usepackage{framed}
\usepackage{svg}
\usepackage{graphicx}
\usepackage[center]{caption}
\usepackage{color,xcolor}
\usepackage{contour}
\usepackage{mathtools}
\mathtoolsset{showonlyrefs}

\providecommand{\NN}{\ensuremath{\mathds{N}}}
\providecommand{\RR}{\ensuremath{\mathbb{R}}}

\providecommand{\ZZ}{\ensuremath{\mathds{Z}}}
\providecommand{\QQ}{\ensuremath{\mathbb{Q}}}

\providecommand{\MM}{\ensuremath{\mathbb{M}_h^d}}
\providecommand{\EE}{\ensuremath{\mathds{E}}}
\providecommand{\VV}{\ensuremath{\mathds{V}}}

\providecommand{\PP}{\ensuremath{\mathds{P}}}
\providecommand{\1}{\ensuremath{\mathbbm{1}}}
\providecommand{\Cyl}{\ensuremath{\mathrm{Cyl}}}

\newtheorem{theorem}{Theorem}
\newtheorem{lemma}[theorem]{Lemma}

\newtheorem{definition}[theorem]{Definition}
\newtheorem{remark}[theorem]{Remark}
\numberwithin{equation}{section}
\numberwithin{theorem}{section}

\title{\textbf{Asymptotics of a time bounded cylinder model}}
\author{Nils Aschenbruck\thanks{Osnabrück University, Institute of Computer Science, D-49076 Osnabrück, Germany, \quad \textbf{aschenbruck@uni-osnabrueck.de}}, Stephan Bussmann and Hanna Döring \thanks{Osnabrück University, Institute of Mathematics, D-49076 Osnabrück, Germany, \quad \textbf{stephan.bussmann@uni-osnabrueck.de, hanna.doering@uni-osnabrueck.de}}}
\date{\today}

\makeatletter
\renewcommand*{\@fnsymbol}[1]{\ensuremath{\ifcase#1\or \circ\or \bullet\or *\or \dagger\or \ddagger\or%
    \cdot\or \times\or \checkmark\or **\or \dagger\dagger%
    \or \ddagger\ddagger \else\textsuperscript{\expandafter\romannumeral#1}\fi}}
\makeatother

\begin{document}
	\maketitle

	\begin{abstract}
		\noindent One way to model telecommunication networks are static Boolean models. However, dynamics such as node mobility have a significant impact on the performance evaluation of such networks. Consider a Boolean model in $\mathbb{R}^d$ and a random direction movement scheme. Given a fixed time horizon $T>0$, we model these movements via cylinders in $\mathbb{R}^d \times [0,T]$. In this work, we derive central limit theorems for functionals of the union of these cylinders. The volume and the number of isolated cylinders and the Euler characteristic of the random set are considered and give an answer to the achievable throughput, the availability of nodes, and the topological structure of the network.
		\\~\\
		\emph{2020 Mathematics Subject Classification:} Primary 60D05, Secondary 60F05
		\\~\\
		\emph{Keywords}: random direction, mobility modeling, central limit theorem,  Poisson cylinder process, stochastic geometry
	\end{abstract}

\section{Introduction}

In the domain of telecommunication network performance evaluation, one of the oldest models used for the dynamics of moving nodes is the random direction mobility model (for reference see \cite{Royer2001} or \cite{Camp2002}).
Nowadays, it is available within standard evaluation frameworks surveyed in \cite{Dede2018}. The random direction model belongs to the group of standard, basic models used in nearly every performance evaluation to understand the elementary properties in well-known standard scenarios before considering more sophisticated and more realistic ones. Within a performance analysis, it is common to analyze series of states of the model indexed by time.
However, for networks that are delay tolerant, such as opportunistic networks, longer periods of time are more interesting.
Opportunistic networks follow the store-carry-forward paradigm, where the movement of the devices is used to carry the messages to their destinations.
In this domain, it is particularly interesting to answer certain questions for a longer period of time such as:
\begin{itemize}
	\item[(1)] Will a message reach a node (at all)?
	\item[(2)] Does a node stay isolated?
\end{itemize}
Since it is clear that mobility changes the properties of a network, there are first results to consider the dynamics of mobile communication networks, see \cite{DMPG09} for example starting with discrete snapshots of the random direction model on the torus. In \cite{DFK16}, the nodes of the telecommunication network move according to the random waypoint model and the connection time of two random nodes is studied. The very recent work \cite{HJC21} also deals with a dynamic Boolean model, where Hirsch, Jahnel and Cali consider two different kinds of movement, small movements with a small velocity and a few large movements with a high velocity. They derive asymptotic formulas for the percolation time and the $k$-hop connection time.\\
In this paper we propose the time bounded cylinder (TBC) model which enables an analysis over a complete timeframe by modeling the movement in $\RR^d$ and the resulting communication capabilities via cylinders in $\RR^d \times [0,T]$. Our first step will be to establish the theory for a rigid movement model where nodes pick a direction and velocity at random and keep these for the complete timeframe. In a second step we will propose a method to introduce changes in direction and velocity.
\\~\\
The following is a slightly simplified construction, which will be made more precise in Section \ref{Section_Construction_of_Cylinders}. We take a set of points in $\RR^d \times \{0\}$ sampled from a homogeneous Poisson point process with intensity $0 < \gamma< \infty$. For each of these points $p_0$ we randomly choose a vector $v$ from the upper half of the $d+1$-dimensional unit sphere. Then we fix some $T\geq 0$ and look at $p_0+v \in \RR^{d+1}$. We scale this vector to a length so that its tip $p_T$ lies in the affine space $\RR^d \times \{T\}$. We fix some radius $r>0$ and for  every point $x$ on this line consider the ball $\mathbb{B}_{d}(x,r)$. The union of these balls defines the time bounded cylinder on $p_0$ in direction $v$. Let $Z$ denote the union of all cylinders created this way for some Poisson process and let $W_s = [-\frac{s}{2}, \frac{s}{2}]^d \times [0,T]$. Our first main result is this:

\begin{theorem}\label{Theorem_Intro_TBC_Vol}
	As $s \to \infty$, the standardized covered volume of the TBC model satisfies
	\begin{equation*}
		\frac{\lambda_{d+1}(Z \cap W_s) - \EE[\lambda_{d+1}(Z \cap W_s)]}{\sqrt{\VV[\lambda_{d+1}(Z \cap W_s)]}} {\mathop{\longrightarrow}^{d}} \mathcal{N}(0,1).
	\end{equation*}
\end{theorem}

We prove this by applying results derived from the Malliavin-Stein method and see that the rate of convergence in the Wasserstein-distance is of order $1/\sqrt{\lambda_{d+1}(W_s)}$, see Theorem \ref{Theorem_TBC_volume_limit_theorem}. Theorem \ref{Theorem_TBC_Volume_Stacked} gives an analogous result in the setting allowing for changes in direction. The covered volume is an interesting quantity in network simulation as it gives some insight into the connectedness of the model. Less volume means more overlap of the communication zones. Therefore, since the TBC model considers the complete timeframe its covered volume actually gives some insight into the achievable throughput of the network.\\
The next property considered is the number of isolated nodes, which means cylinders that do not intersect with the rest of the model within the timeframe.

\begin{theorem}\label{Theorem_Intro_TBC_Iso}
Let $\mathrm{Iso_{Z}}(s)$ denote the number of cylinders with basepoint in $[-\frac{s}{2}, \frac{s}{2}]^{d}$ that do not intersect with any cylinder in the TBC model. Then
\begin{equation*}
	\frac{\mathrm{Iso_{Z}}(s) - \EE[\mathrm{Iso_{Z}}(s)]}{\sqrt{\VV[\mathrm{Iso_{Z}}(s)]}} {\mathop{\longrightarrow}^{d}} \mathcal{N}(0,1)
\end{equation*}
as $s \to \infty$.
\end{theorem}

As before we additionally offer a rate of convergence (see Theorem \ref{Theorem_TBC_isolated_nodes_limit_theorem}) and a version for changeable directions (see Theorem \ref{Theorem_TBC_Iso_Stacked}). These nodes are of particular interest in opportunistic networks as one can not communicate with them. Let us assume a sensor node (e.g. deployed for wildlife monitoring \cite{Anthony2012}) that can store the sensor data for some time, but has to transmit the data before it runs out of memory. If this node stays isolated for too long, it results in loss of data (e.g. inaccuracies in the wildlife tracking trace).
Lastly we come to a quantity of more theoretical significance and present a limit theorem for the Euler characteristic of the model, a precise definition of which is given in Section \ref{Section_Main_results_and_proofs_Cylinders}.
\begin{theorem}\label{Theorem_Intro_TBC_Euler}
 	For $s \to \infty$ the Euler characteristic $\chi$ of the TBC model $Z$ restricted to a window $W_s=[-\frac{s}{2}, \frac{s}{2}]^{d}$ satisfies a central limit theorem, namely
	\begin{equation*}
		\frac{\chi(Z \cap W_s) - \EE[\chi(Z \cap W_s)]}{\sqrt{\VV[\chi(Z \cap W_s)]}} {\mathop{\longrightarrow}^{d}} \mathcal{N}(0,1).
	\end{equation*}
\end{theorem}
Since the Euler characteristic can be defined as the alternating sum of Betti numbers, we see this result as an incentive for further mathematical studies of the Betti numbers of the TBC model. These would give great insight into the topological structure of the model, such as the number of connected components or the number of holes. From a computer science perspective, both of these would be of special interest.

\section{Construction of the Cylinders}\label{Section_Construction_of_Cylinders}

We consider a stationary Poisson process $\eta$ in $\RR^d$ with intensity $\gamma \in (0, \infty)$. We will refer to its points as the \textbf{basepoints} of the cylinders we are about to construct. In terms of random networks think of these points as the positions of our nodes at time $0$. 
Next, let $h \in (0,1)$ denote a real constant and define

\begin{equation*}
	\MM := \{ v \in S^d : v_{d+1} > h\}
\end{equation*}
and let $\QQ$ denote some probability measure on $\MM$. We shall refer to vectors chosen according to this measure as the \textbf{directions} of our cylinders. $\QQ$ may be discrete or continuous, we do not impose any restrictions on it other than not being allowed to depend on the position of the basepoints. As long as the direction is chosen independently and identically for all basepoints the arguments made in this paper will hold.

\begin{figure}
	\begin{center}
		\includegraphics[scale=0.12]{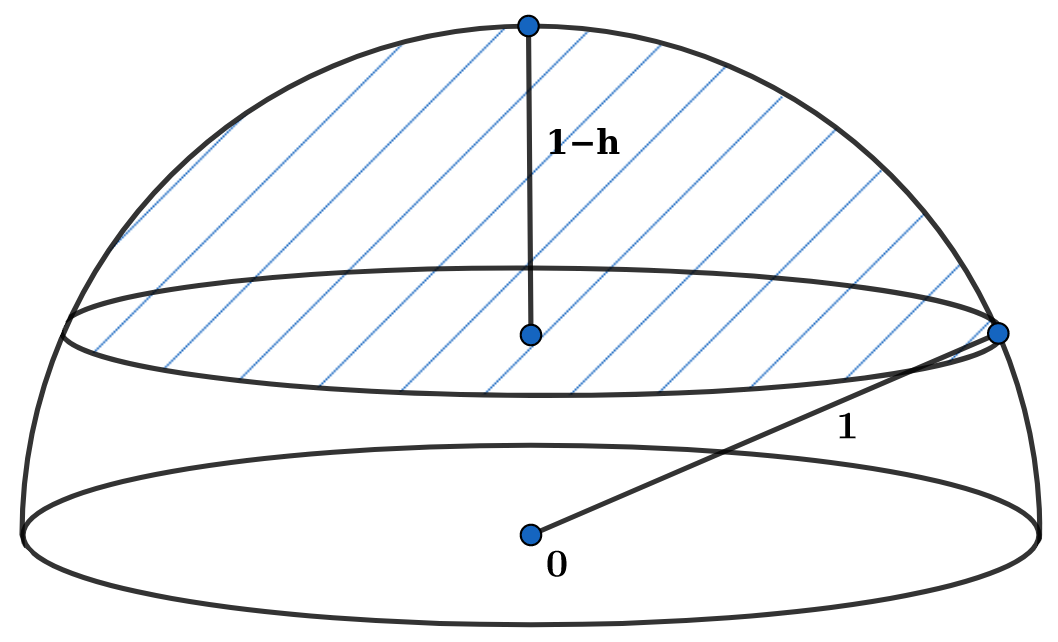}
	\end{center}
	\caption{The blue shaded area is the set of possible directions when $d=2$.}
\end{figure}

To construct a cylinder fix some $T \in [0,\infty)$. We call this parameter the \textit{time horizon} of the model. For $(x_1,\dots,x_d) \in \RR^d$ and $u \in \RR$ we set $\hat{x}_u = (x_1,\dots,x_d,u) \in \RR^{d+1}$. Given a basepoint $p \in \RR^d$ and a direction $v \in \MM$, we define a time bounded cylinder as a union of Minkowski sums. $\mathcal{C}^{d}$ shall denote the system of compact subsets of $\RR^{d}$.
\begin{definition}\label{Definition_Cylinder}
	Given a radius $r \geq 0$ and a time horizon $T$, the value of the function
	\begin{align*}
		\mathrm{Cyl}_{r,T}:~\RR^d \times \MM &\to \mathcal{C}^{d+1}\\
		(p,v) &\mapsto \bigcup_{t \in [0,1]} \left(\Big(\frac{tT}{v_{d+1}} \cdot v + \hat{p}_0\Big) + \mathbb{B}_d(r)\right)
	\end{align*}
	is called a \textbf{time bounded cylinder}.
\end{definition}

The first $d$ entries of a vector $v \in \MM$ thus encode the direction a node is moving in. The entry with index $d+1$ encodes the velocity at which that node is traveling. We can now see that $v_{d+1}=1$ means a node remains in place, while smaller values for $v_{d+1}$ imply faster movement. $v_{d+1}=0$ would imply infinite velocity which, in turn, would imply cylinders of infinite length. Thus we need the restriction $v_{d+1} > h$ to limit the movement speed of nodes.\\
Since $r$ and $T$ are constant in our model, we define $\mathrm{Cyl}(p,v) := \mathrm{Cyl}_{r,T}(p,v)$ for ease of notation. It will prove fruitful to think of the cylinders themselves as elements of a Poisson point process. To that end, we take the points of $\eta$ and attach the directions as markings, cf. \cite[Definition 5.3]{Last2018}.
\begin{definition}\label{Definition_Process_of_Cylinders}
	Let $\eta$ and $\QQ$ be defined as above. Mark every point in the support of $\eta$  with a direction from $\MM$, randomly chosen i.i.d. according to $\QQ$. The resulting marked point process $\mathbf{\xi}$, defined on $\RR^d \times \MM$, is called the \textbf{process of time bounded cylinders}.
\end{definition}

\begin{lemma}\label{Lemma_Xi_is_a_Poisson_Point_Process}
	$\xi$ is a Poisson point process with intensity measure $\Lambda= \gamma \lambda_d \otimes \QQ$.
\end{lemma}
\begin{proof}
	By Definition \cite[Definition 5.3]{Last2018}, $\xi$ is an independent $\QQ$-marking. The asserted statement now follows directly from the Marking Theorem \cite[Theorem 5.6]{Last2018}.
\end{proof}

We are now in a position to define our model:

\begin{definition}\label{Definition_TBC_model}
	With the point process $\xi$ given as above, the \textbf{time bounded cylinder (TBC) model} is the random set
	\begin{equation*}
		Z(\xi) := \bigcup_{(p,v) \in \xi} \mathrm{Cyl}(p,v).
	\end{equation*}
	In this notation we identify the random measure $\xi$ with its support.
\end{definition}
\begin{figure}
	\begin{center}
		\includegraphics[scale=0.6255]{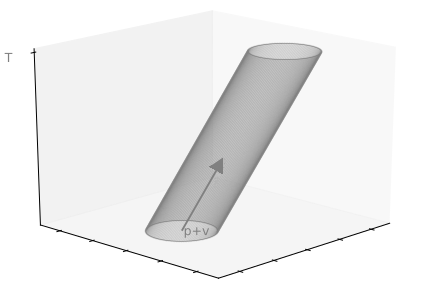}\includegraphics[scale=0.5]{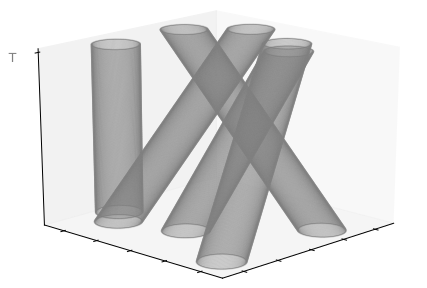}
	\end{center}
	\caption{In the left picture we see a cylinder with basepoint $p$ and direction $v$. The arrow marks the displaced vector $\hat{p}_0+v$. To the right we have an excerpt of a TBC model $Z$ constructed on $\RR^2$.}
\end{figure}

This model is a modification of the Poisson cylinder model (see e.g. \cite{Weil1987}) where a cylinder in $\RR^d$ is constructed by choosing a $q$-dimensional, $0 \leq q \leq d-1$, linear subspace of $\RR^d$ and taking the Minkowksi sum of this subspace and a set taken from its orthogonal complement. A Poisson cylinder process is then simply a Poisson process on the space of these cylinders. The union $Z$ of the cylinders created by such a process in the stationary case has been subject to recent study. In \cite{Hein2009} Heinrich and Spiess derive a central limit theorem for the $d$-dimensional volume of $Z$ restricted to a growing window of observation by taking into account the long-range dependencies and applying the method of cumulants.
In \cite{Hein2013} the authors modified that limit theorem using weakened assumptions on the cylinder bases and proved a new central limit theorem for the surface content of the model. In \cite{Baci2019} Baci et al. were able to derive concentration inequalities for the volume and the intrinsic volumes of the model, again restricted to a compact window of observation. Assumptions of convexity on the cylinder bases and isotropy of $Z$ were made. In the very recent preprint \cite{Betken2021} Betken et al. show central limit theorems for geometric funtionals for Poisson cylinder processes. The special case of $q=1$ and spherical cylinder bases is studied in \cite{Broman2016}, where  Broman and Tykesson focus on connectivity properties. In the preprint \cite{Flimmel2021}, Flimmel and Heinrich also followed the idea of cylinders constructed by marked point processes, but they restrict the underlying point process to lie on the real line. They determine the asymptotic limit for the variance of the covered area for stripes of random thickness and prove a law of large numbers.\\
The classical Poisson cylinder model has two properties that do not fit to the telecommunication network: The cylinders at a fixed time step would imply ellipses as a communication zone depending on individual speed and direction -- smaller angles (representing high speed) imply higher eccentricity. In a telecommunication network however, the spherical communication area is preserved at any time step independently of the movement. Secondly almost horizontally running cylinders as in the classical Poisson cylinder model would mean that one travels with almost infinite speed. The time bounded cylinder model takes up these two aspects.
Since for fast movements the communication zone resembles flat tubes rather than cylinders which has direct influence on the covered volume of their intersection. On the other hand the cylinders which cross almost horizontally in the classical Poisson cylinder model have an impact on the connection of the cylinders because they connect many other cylinders at once. This is why we consider it interesting to study the asymptotics of the modified time bounded cylinder model.

\section{Preliminaries}

A helpful tool in the study of random closed sets in general and in particular when proving the main results presented in this chapter are capacity functionals. These can be seen as analogues to the distribution functions of real-valued random variables for random closed sets, see \cite{Schneider2008} for reference.
\begin{definition}\label{Definition_Capacity_Functional}
	The \textbf{capacity functional} of a random closed set $Z \subset \RR^d$ is a mapping defined by the probability
	\begin{equation*}
		\PP(Z \cap C \neq \emptyset)
	\end{equation*}
	with argument $C \in \mathcal{C}^d$.
\end{definition}

To derive the capacity functional of our model we will need the following definition.

\begin{definition}\label{Definition_TBC_v_Shadow}
	Given a set $C \in \RR^d \times [0,T]$ and a direction $v \in \MM$ we denote by $\overline{C}_v$ the $\mathbf{v}$\textbf{-shadow} of $C$ onto $\RR^d$, that is the set
	\begin{equation*}
		\overline{C}_v := \Big\{ x - \frac{x_{d+1}}{v_{d+1}}v ~|~ x \in C\Big\} \subset \RR^d \times \{0\}.
	\end{equation*}
\end{definition}
Intuitively speaking, the $v$-shadow of a point $x \in \RR^{d+1}$ is the point of intersection of the line $\{x+s\cdot v ~|~ s \in \RR\}$ and $\RR^d \times \{0\}$. That is the projection of $x$ on $\RR^d \times \{0\}$ along $v$. With this we can formulate our next lemma which gives a concise form of the capacity functional for the TBC model.

\begin{lemma}\label{Lemma_TBC_Capacity Functional}
	For $Z(\xi)$ as in Definition \ref{Definition_TBC_model} and a compact $C \subset \RR^d \times [0,T]$ we have
	\begin{equation*}
		\PP(Z(\xi) \cap C \neq \emptyset) = 1 - \exp\Big(-\gamma \cdot \int_{\MM} \lambda_d\big(\overline{C}_v + \mathbb{B}_d(r)\big) ~\QQ(\mathrm{d}v)\Big).
	\end{equation*}
\end{lemma}

\begin{proof}
	Define the set
	\begin{equation*}
		A_C := \{(p,v) \in \RR^d \times \MM ~|~ \mathrm{Cyl}(p,v) \cap C \neq \emptyset \}.
	\end{equation*}
	By Lemma \ref{Lemma_Xi_is_a_Poisson_Point_Process} $\xi$ is a Poisson point process. Thus we have that
	\begin{equation*}
		\PP(Z(\xi) \cap C \neq \emptyset) = 1 - \PP(\xi(A_C) = 0) = 1 - \exp(-\Lambda(A_C)).
	\end{equation*}
	Note that for any fixed $v \in \MM$ and $x \in \RR^d \times [0,T]$ there exists a ball of radius $r$ in $\RR^d$ such that a cylinder with direction $v$ will include $x$ if and only if its basepoint is included in that ball. The union of these balls for all points $x \in C$ coincides with the set $\overline{C}_v + \mathbb{B}_d(r)$. We get
	\begin{align*}
		&\Lambda(A_C)\\ & = \gamma \int_{\MM} \int_{\RR^d} \mathbbm{1}(\mathrm{Cyl}(p,v) \cap C \neq \emptyset) ~\lambda_d(\mathrm{d}p) \QQ(\mathrm{d}v) \\
		&= \gamma \int_{\MM} \int_{\overline{C}_v + \mathbb{B}_d(r)} 1 ~\lambda_d(\mathrm{d}p) \QQ(\mathrm{d}v) = \gamma \int_{\MM} \lambda_d\big(\overline{C}_v + \mathbb{B}_d(r)\big) ~\QQ(\mathrm{d}v).
	\end{align*}
\end{proof}

\noindent If one wants to express the probability that a particular point $x \in \RR^d \times [0,T]$ gets hit by the set of cylinders, then Lemma \ref{Lemma_TBC_Capacity Functional} with $C := \{x\}$ yields
\begin{equation}\label{Equation_Probability_x_hits_TBC}
	\PP(x \in Z(\xi)) = 1 - \exp(-\gamma \cdot r^d \kappa_d)
\end{equation}
where $\kappa_d$ denotes the volume of the $d$-dimensional unit ball. To give bounds on the capacity functional we sometimes need the following definition.
\begin{definition}\label{Defininition_Cylinder_Scope}
	Given a pair $(p,v) \in \RR^d \times \MM$, we call the distance from basepoint $(p_1,\dots,p_d,0)$ to
	endpoint $\frac{T}{v_{d+1}}(v_1,\dots,v_{d+1}) + (p_1,\dots,p_d,0)$,
	\begin{equation*}
		R_{v} := \biggl\|\frac{T}{v_{d+1}}(v_1,\dots,v_{d+1})\biggl\| = \frac{T}{v_{d+1}},
	\end{equation*}
	the \textbf{scope} of the corresponding cylinder $\mathrm{Cyl}(p,v)$. We denote the maximum scope of cylinders in our model by $R_h := \frac{T}{h}$.
\end{definition}
Think of the scope of a cylinder as the radius of its $d$-dimensional area of influence. For $(p,v),(q,w) \in \RR^d \times \MM$ the cylinders $\mathrm{Cyl}(p,v)$ and $\mathrm{Cyl}(q,w)$ can intersect if and only if $\|p-q\| \leq R_v + R_w + 2r \leq 2(R_h+r)$. Now we want to estimate $\PP(\mathrm{Cyl}(p,v) \cap Z(\xi) \neq \emptyset)$. Notice that
$\lambda_d\big(\overline{\mathrm{Cyl}(p,v)_w} + \mathbb{B}_d(r)\big)$ takes its minimum in the case $w = v$. Then we have that $\overline{\mathrm{Cyl}(p,v)}_w = \mathbb{B}_d(p,r)$. The maximum is reached when $v$ and $w$ have opposite directions while $w_{d+1} = h$. Then it follows that
\begin{align*}
	\lambda_d\big(\overline{\mathrm{Cyl}(0,v)_{w}} + \mathbb{B}_d(r)\big) &\leq 2(R_{h}+r) \cdot \lambda_{d}\big(\mathbb{B}_d(2r)\big) < \infty
\end{align*}
for all $v \in \MM$. We get
\begin{equation*}
	0 < (2r)^d \kappa_d \leq \lambda_d\big(\overline{\mathrm{Cyl}(p,v)_w} + \mathbb{B}_d(r)\big) \leq (R_{h}+r)2^{d+1}r^d \kappa_d
\end{equation*}
and thus by Lemma \ref{Lemma_TBC_Capacity Functional}
\begin{equation}\label{Equation_Probability_Cylinder_hits_TBC}
	1 - \exp(-\gamma (2r)^d \kappa_d) \leq \PP(\mathrm{Cyl}(p,v) \cap Z(\xi) \neq \emptyset) \leq 1 - \exp(-\gamma (R_{h}+r)2^{d+1}r^d\kappa_d).
\end{equation}

\section{Applying the Malliavin-Stein Bound}\label{Section_Applying_the_Malliavin_Stein_Bound}

We follow the general definition in \cite[Chapter 18]{Last2018} and consider a measurable space $(\mathbb{X}, \mathcal{X})$ and denote by $\mathbf{N}$ the space of all locally finite counting measures on $\mathbb{X}$. Let $\eta$ denote a Poisson process on $\mathbb{X}$ with $\sigma$-finite intensity measure $\Lambda$.

\begin{definition}\label{Definition_Poisson_Functional}
		A random variable $F$ such that $F = f(\eta)$ $\PP$-a.s. for some measurable function $f: \mathbf{N} \to \RR$ is called a \textbf{Poisson functional}. In this notion $f$ is called a representative of $F$.
\end{definition}

Our aim in this paper is to derive asymptotic distributions of Poisson functionals based on $\xi$. This involves the following operator which can be seen as a discrete analogue to classical derivatives.

\begin{definition}\label{Definition_Difference_Operator}
	For $x \in \mathbb{X}$ we define a map $\mathrm{D}_xf: \mathbf{N} \to \RR$ by
	\begin{equation*}
		\mathrm{D}_xf := f(\mu + \delta_x) - f(\mu), \quad \mu \in \mathbf{N}.
	\end{equation*}
	This map is called the \textbf{difference} (or \textbf{add-one-cost}) \textbf{operator} of $f(\mu)$. We can extend this definition inductively for $n \geq 2$ and $(x_1, \dots, x_n) \in \mathbb{X}^n$ by setting $\mathrm{D}^1 := \mathrm{D}$ and defining $\mathrm{D}_{x_1, \dots, x_n}^nf : \mathbf{N} \to \RR$ by
	\begin{equation*}
		\mathrm{D}_{x_1,\dots,x_n}^nf := \mathrm{D}_{x_1}^1 \mathrm{D}_{x_2,\dots,x_n}^{n-1}f.
	\end{equation*}
	We then call $\mathrm{D}_{x_1,\dots,x_n}^nf$ the $\mathbf{n}$-\textbf{th order difference operator}.
\end{definition}

The following well-known theorem is derived from the Malliavian-Stein method and uses difference operators to give a bound on the Wasserstein distance $\mathrm{d_W}$ of the law of a Poisson functional and the standard normal distribution.

\begin{theorem}[{\cite[Theorem 21.3]{Last2018}}]\label{Theorem_LP18_Limit_Theorem}
	Let $\eta$ denote a Poisson process with intensity measure $\Lambda$. Consider a Poisson functional $F$ on $\eta$ with $\EE[F] = 0$ and $\VV[F]=1$. $N$ shall denote a standard normal random variable. If $\EE\big[\int (\mathrm{D}_x F)^2 \Lambda(\mathrm{d}x) \big] < \infty$ then there are constants
	\begin{align*}
		\alpha_1(F) &= \int \sqrt{\EE[(\mathrm{D}_xF)^2 (\mathrm{D}_yF)^2]} \cdot \sqrt{\EE[(D_{x,z}^2F)^2 (D_{y,z}^2F)^2]} ~\Lambda^3(\mathrm{d}(x,y,z)),\\
		\alpha_2(F) &= \int \EE[(\mathrm{D}_{x,z}^2F)^2(\mathrm{D}_{y,z}^2F)^2]~\Lambda^3(\mathrm{d}(x,y,z)),\\
		\alpha_3(F) &= \int \EE[|\mathrm{D}_xF|^3] ~\Lambda(\mathrm{d}x),
	\end{align*}
	such that
	\begin{equation}
		\mathrm{d_W}(F,N) \leq 2 \sqrt{\alpha_1(F)} + \sqrt{\alpha_2(F)} + \alpha_3(F).
	\end{equation}
\end{theorem}

We will now apply this theorem to our setting.

\begin{lemma}\label{Lemma_Apply_Malliavin_Stein_Bound}
	Let $f: \mathcal{C}^{d+1} \to \RR$, $s\geq1$ and $W_s = [-\frac{s}{2}, \frac{s}{2}]^{d} \times [0,T]$. Assume that $|f(A)|<\infty$ for all $A \in \mathcal{C}^{d+1}$ and that there exist constants $c_1$, $c_2$, $c_3$ and $R \in \RR$ such that the following conditions are met for all $x = (p,v), y= (q,w) \in  \RR^d \times \MM$:
	\begin{itemize}
		\item[$\mathrm{(A)}$] $\mathrm{D}_xf(Z(\xi) \cap W_s) = 0\quad \PP \text{-a.s.}$ \text{for} $\|p\| > R+s$,
		\item[$\mathrm{(B)}$] $\VV\big[f(Z(\xi) \cap W_s)\big] \geq c_1 \cdot \lambda_{d+1}(W_s)$,\vspace*{0.1cm}
		\item[$\mathrm{(C)}$] $\max\{\EE\big[\mathrm{D}_{x}f(Z(\xi) \cap W_s)^4\big],~\EE\big[\mathrm{D}_{x}f(Z(\xi+\delta_y) \cap W_s)^4\big]\} \leq c_2$,\vspace*{0.1cm}
		\item[$\mathrm{(D)}$] $\|p - q\|>R \Rightarrow \EE[\big(\mathrm{D}_{x,y}^2 f(Z(\xi) \cap W_s)\big)^4] \leq \frac{c_3}{\lambda_{d+1}(W_s)^4}$.
	\end{itemize}\vspace*{0.25cm}
	Then, there is a constant $c \in \RR$ such that, for a standard normal random variable $N$,
	\begin{equation*}
		\mathrm{d_W}\Bigl(\frac{f(Z(\xi) \cap W_s)-\EE\big[f(Z(\xi) \cap W_s)\big]}{\sqrt{\VV\big[f(Z(\xi) \cap W_s)\big]}}, N\Bigr) \leq \frac{c}{\sqrt{\lambda_{d+1}(W_s)}}.
	\end{equation*}
\end{lemma}

The proof of this lemma gives insight into the composition of $c$, which is rather complex. When proving our main results however, we will always set $c_3=0$. In that case the constant is reduced to the following:

\begin{remark}
	Let $c_3=0$. In that case the constant $c$ in Lemma \ref{Lemma_Apply_Malliavin_Stein_Bound} is given by
	\begin{equation*}
		c = \frac{\gamma\sqrt{c_{d,R,T}}}{c_1^\frac{3}{2}}\left(8\sqrt{\gamma c_1 c_2} \lambda_d(\mathbb{B}_d(R)) + (1+c_2)\sqrt{c_{d,R,T}} \right)
	\end{equation*}
	where $c_{d,R,T} = \frac{2^{2d-1}(R^d+\frac{1}{2^d})}{T}$. Note however, that $R$ will depend on $T$.
\end{remark}
An interesting fact is that the assumptions in the lemma imply an upper bound on the variance, which is of the same order as the lower bound.
\begin{remark}\label{Remark_TBC_Upper_Bound_Variance}
	Since (C) implies $\EE[\mathrm{D}_{x}f(Z(\xi) \cap W_s)^2] \leq 1+c_2$, the Poincaré inequality gives us the upper bound
	\begin{equation*}
		\VV\big[f(Z(\xi) \cap W_s)\big] \leq \int_{W_s \times \MM} \EE\big[\mathrm{D}_{x}f(Z(\xi) \cap W_s)^2\big] \Lambda(\mathrm{d}x) \leq (1+c_2) \cdot \lambda_{d+1}(W_s).
	\end{equation*}
\end{remark}
\bigskip
\begin{proof}[Proof of Lemma \ref{Lemma_Apply_Malliavin_Stein_Bound}]
	We start by defining
	\begin{equation*}
		g(\xi) := \frac{f\big(Z(\xi) \cap W_s\big) - \EE[f\big(Z(\xi) \cap W_s\big)]}{\sqrt{\VV[f\big(Z(\xi) \cap W_s\big)]}}
	\end{equation*}
	and aim to use Theorem \ref{Theorem_LP18_Limit_Theorem} on the Poisson functional $G$ with representative $g$. First let us check if the prerequisites are met. From the finiteness of $f$ and condition (B) it follows that $\EE[|G|^2] < \infty$. As we have seen before, as a function of $f$ the difference operator is linear. This $\PP \text{-a.s.}$ leads to
	\begin{equation}\label{Equation_Difference_Operator_of_Standardization_of_G_Linear}
		|\mathrm{D}_{x}G| = \frac{|\mathrm{D}_{x}f(Z(\xi) \cap W_s)|}{\sqrt{\VV[f(Z(\xi) \cap W_s)]}}.
	\end{equation}
	It follows now from conditions (A) and (B) that $\EE\big[\int (\mathrm{D}_x G)^2 \Lambda(\mathrm{d}x) \big] < \infty$. Thus we can apply the theorem and must now derive bounds for $\alpha_1(G)$, $\alpha_2(G)$ and $\alpha_3(G)$. So let $x,y \in \RR^d \times \MM$. By the Cauchy-Schwarz inequality, \eqref{Equation_Difference_Operator_of_Standardization_of_G_Linear} and stationarity of the Poisson process we have
	\begin{align}\label{Equation_First_Order_Difference_Operator_Bound_Condition_C}
		\EE[(\mathrm{D}_{x}G)^2 (\mathrm{D}_{y}G)^2] &\leq \sqrt{\EE[(\mathrm{D}_{x}G)^4]} \sqrt{\EE[(\mathrm{D}_{y}G)^4]}\nonumber\\
		&= \frac{\EE[\big(\mathrm{D}_{x}f(Z(\xi) \cap W_s)\big)^4]}{\sqrt{\VV[f(Z(\xi) \cap W_s)]}^4} \nonumber\\
		&\leq \frac{c_2}{c_1^2 \cdot \lambda_{d+1}(W_s)^2}
	\end{align}
	where the last inequality follows from conditions (B) and  (C). To get a bound on the second order operators, we use the well-known fact that for $a,b \in \RR$ and $q\geq1$ the inequality $|a+b|^q \leq 2^{q-1}(|a|^q+|b|^q)$ holds. Additionally, we use the linearity of the difference operator and then (B) and (C) again.
	\begin{align}\label{Equation_Second_Order_Difference_Operator_Bound_Condition_C}
		\EE[|\mathrm{D}^2_{x,y}G|^4] &= \frac{\EE[|\mathrm{D}_{x,y}^2f\big(Z(\xi) \cap W_s)|^4]}{\VV[f(Z(\xi) \cap W_s)]^2}\nonumber \\
		&\leq \frac{2^3 \cdot\big(\EE[|\mathrm{D}_{x}f\big(Z(\xi + \delta_y) \cap W_s)\big)|^4]+\EE[|\mathrm{D}_{x}f(Z(\xi) \cap W_s)|^4]\big)}{c_1^2 \cdot \lambda_{d+1}(W_s)^2}\nonumber\\
		&\leq \frac{2^4 \cdot c_2}{c_1^2 \cdot \lambda_{d+1}(W_s)^2}.
	\end{align}
	Condition (D) sharpens this estimate to
	\begin{align}\label{Equation_Second_Order_Difference_Operator_Bound_Condition_D}
		\EE[|\mathrm{D}^2_{x,y}G|^4] &= \frac{\EE[|\mathrm{D}_{x,y}^2f\big(Z(\xi) \cap W_s)|^4]}{\VV[f(Z(\xi) \cap W_s)]^2}
		\leq \frac{c_3}{c_1^2 \cdot \lambda_{d+1}(W_s)^6}.
	\end{align}
	Now let $\overline{W}_{R,s} = [-\frac{s}{2}, \frac{s}{2}]^d + \mathbb{B}_d(R)$. It follows from condition (A), \eqref{Equation_First_Order_Difference_Operator_Bound_Condition_C} and Cauchy-Schwarz that
	\begin{align*}
		\alpha_1(G) &= \int_{(\RR^d \times \MM)^3} \sqrt{\EE[(\mathrm{D}_{x}G)^2 (\mathrm{D}_{y}G)^2]}\sqrt{\EE[(\mathrm{D}_{x,z}^2G)^2 (\mathrm{D}_{y,z}^2G)^2]} ~\Lambda^3\big(\mathrm{d}(x,y,z)\big)\\
		 &= \int_{(\overline{W}_{R,s} \times \MM)^3} \sqrt{\EE[(\mathrm{D}_{x}G)^2 (\mathrm{D}_{y}G)^2]}\sqrt{\EE[(\mathrm{D}_{x,z}^2G)^2 (\mathrm{D}_{y,z}^2G)^2]} ~\Lambda^3\big(\mathrm{d}(x,y,z)\big)\\
		 &\leq \int_{(\overline{W}_{R,s} \times \MM)^3} \frac{\sqrt{c_2}}{c_1 \cdot \lambda_{d+1}(W_s)} \cdot \sqrt{\EE[(\mathrm{D}_{x,z}^2G)^2 (\mathrm{D}_{y,z}^2G)^2]}~\Lambda^3 \big(\mathrm{d}(x,y,z)\big) \\
		&\leq \int_{(\overline{W}_{R,s} \times \MM)^3} \frac{\sqrt{c_2}}{c_1 \cdot \lambda_{d+1}(W_s)} \cdot \sqrt[4]{\EE[\big(\mathrm{D}_{x,z}^2G\big)^4]}\sqrt[4]{\EE[\big(\mathrm{D}_{y,z}^2G\big)^4]}~\Lambda^3 \big(\mathrm{d}(x,y,z)\big).
	\end{align*}
	We will now disassemble the domain of integration in order to properly apply our conditions. To that end, set $A_x = \mathbb{B}_d(x_1,R)\times \MM$ and $A_x^C = \overline{W}_{R,s}\backslash\mathbb{B}_d(x_1,R) \times \MM$. This gives us the decomposition
	\begin{align}\label{Equation_Alpha_1_Decomposition}
		&\int_{(\overline{W}_{R,s} \times \MM)^3} \sqrt[4]{\EE[\big(\mathrm{D}_{x,z}^2G\big)^4]}\sqrt[4]{\EE[\big(\mathrm{D}_{y,z}^2G\big)^4]}~\Lambda^3 \big(\mathrm{d}(x,y,z)\big)\\
		=& \int_{\overline{W}_{R,s} \times \MM} \int_{A_z} \int_{A_z} \sqrt[4]{\EE[\big(\mathrm{D}_{x,z}^2G\big)^4]}\sqrt[4]{\EE[\big(\mathrm{D}_{y,z}^2G\big)^4]} ~\Lambda \big(\mathrm{d}x\big)~\Lambda \big(\mathrm{d}y\big) ~\Lambda \big(\mathrm{d}z\big)\nonumber\\
		&+ \int_{\overline{W}_{R,s} \times \MM} \int_{A_z} \int_{A_z^C} \sqrt[4]{\EE[\big(\mathrm{D}_{x,z}^2G\big)^4]}\sqrt[4]{\EE[\big(\mathrm{D}_{y,z}^2G\big)^4]} ~\Lambda \big(\mathrm{d}x\big)~\Lambda \big(\mathrm{d}y\big) ~\Lambda \big(\mathrm{d}z\big)\nonumber\\
		&+ \int_{\overline{W}_{R,s} \times \MM} \int_{A_z^C} \int_{A_z} \sqrt[4]{\EE[\big(\mathrm{D}_{x,z}^2G\big)^4]}\sqrt[4]{\EE[\big(\mathrm{D}_{y,z}^2G\big)^4]} ~\Lambda \big(\mathrm{d}x\big)~\Lambda \big(\mathrm{d}y\big) ~\Lambda \big(\mathrm{d}z\big)\nonumber\\
		&+ \int_{\overline{W}_{R,s} \times \MM} \int_{A_z^C} \int_{A_z^C} \sqrt[4]{\EE[\big(\mathrm{D}_{x,z}^2G\big)^4]}\sqrt[4]{\EE[\big(\mathrm{D}_{y,z}^2G\big)^4]} ~\Lambda \big(\mathrm{d}x\big)~\Lambda \big(\mathrm{d}y\big) ~\Lambda \big(\mathrm{d}z\big).\nonumber
	\end{align}
	The first summand on the right hand side of \eqref{Equation_Alpha_1_Decomposition} can now be bounded using \eqref{Equation_Second_Order_Difference_Operator_Bound_Condition_C}. This yields
	\begin{align*}
		&\int_{\overline{W}_{R,s} \times \MM} \int_{A_z} \int_{A_z} \sqrt[4]{\EE[\big(\mathrm{D}_{x,z}^2G\big)^4]}\sqrt[4]{\EE[\big(\mathrm{D}_{y,z}^2G\big)^4]} ~\Lambda \big(\mathrm{d}x\big)~\Lambda \big(\mathrm{d}y\big) ~\Lambda \big(\mathrm{d}z\big)\\
		&\leq \lambda_{d}(\overline{W}_{R,s}) \lambda_{d}\big(\mathbb{B}_d(R)\big)^2 \frac{4\gamma^3 \sqrt{c_2}}{c_1 \cdot \lambda_{d+1}(W_s)}.
	\end{align*}
	By Fubini's Theorem, the second and third summand are the same. Note that the implication in condition $(D)$ applies to all pairs $(x,w)$ with $w \in A_x^C$, since the required distance is met. Thus we can bound these summands by using \eqref{Equation_Second_Order_Difference_Operator_Bound_Condition_C} and \eqref{Equation_Second_Order_Difference_Operator_Bound_Condition_D} which gets us
	\begin{align*}
		&2\cdot\int_{\overline{W}_{R,s} \times \MM} \int_{A_z} \int_{A_z^C} \sqrt[4]{\EE[\big(\mathrm{D}_{x,z}^2G\big)^4]}\sqrt[4]{\EE[\big(\mathrm{D}_{y,z}^2G\big)^4]} ~\Lambda \big(\mathrm{d}x\big)~\Lambda \big(\mathrm{d}y\big) ~\Lambda \big(\mathrm{d}z\big)\\
		&\leq 2\gamma^3 \cdot\lambda_{d}\big(\mathbb{B}_d(R)\big)\lambda_{d}(\overline{W}_{R,s})^2 \frac{2\sqrt[4]{c_2}}{\sqrt{c_1} \cdot \sqrt{\lambda_{d+1}(W_s)}}\frac{\sqrt[4]{c_3}}{\sqrt{c_1} \cdot \lambda_{d+1}(W_s)^\frac{3}{2}}.
	\end{align*}
	Another application of \eqref{Equation_Second_Order_Difference_Operator_Bound_Condition_D} gives us a bound on the last summand.
	\begin{align*}
		&\int_{\overline{W}_{R,s} \times \MM} \int_{A_z^C} \int_{A_z^C} \sqrt[4]{\EE[\big(\mathrm{D}_{x,z}^2G\big)^4]}\sqrt[4]{\EE[\big(\mathrm{D}_{y,z}^2G\big)^4]} ~\Lambda \big(\mathrm{d}x\big)~\Lambda \big(\mathrm{d}y\big) ~\Lambda \big(\mathrm{d}z\big) \\
		&\leq \lambda_{d}(\overline{W}_{R,s})^3 \frac{\gamma^3\sqrt{c_3}}{c_1 \cdot \lambda_{d+1}(W_s)^3}.
	\end{align*}
	Next we note that for the volume of $\overline{W}_{R,s}$ we have
	\begin{align}\label{Equation_Volume_Observation_Window_Upper_Bound}
		\lambda_{d}(\overline{W}_{R,s}) &\leq (2R+s)^d \leq 2^{d-1}(2^dR^d+s^d) \leq 2^{d-1}(2^dR^d+1)s^d\\
		&= \frac{2^{2d-1}(R^d+\frac{1}{2^d})}{T}\lambda_{d+1}(W_s) =: c_{d,R,T}\cdot\lambda_{d+1}(W_s)  .\nonumber
	\end{align}
	Using these results, we can now further estimate $\alpha_1(G)$ by
	\begin{align*}
		\alpha_1(G) \leq& \frac{\sqrt{c_2}}{c_1 \cdot \lambda_{d+1}(W_s)} \cdot \int_{(\overline{W}_{R,s} \times \MM)^3} \sqrt[4]{\EE[\big(\mathrm{D}_{x,z}^2G\big)^4]}\sqrt[4]{\EE[\big(\mathrm{D}_{y,z}^2G\big)^4]}~\Lambda^3 \big(\mathrm{d}(x,y,z)\big)\\
		\leq& \lambda_{d}(\overline{W}_{R,s}) \lambda_{d}\big(\mathbb{B}_d(R)\big)^2 \frac{4\gamma^3c_2}{c_1^2 \cdot \lambda_{d+1}(W_s)^2}\\
		&+ 2\cdot\lambda_{d}\big(\mathbb{B}_d(R)\big)\lambda_{d}(\overline{W}_{R,s})^2 \frac{2\gamma^3c_2^\frac{3}{4}\sqrt[4]{c_3}}{c_1^2 \cdot \lambda_{d+1}(W_s)^3}\\
		&+ \lambda_{d}(\overline{W}_{R,s})^3 \frac{\gamma^3\sqrt{c_2}\sqrt{c_3}}{c_1^2 \cdot \lambda_{d+1}(W_s)^4}\\
		=& \frac{\gamma^3 c_{d,R,T}\sqrt{c_2}}{c_1^2 \cdot \lambda_{d+1}(W_s)}\big(4\sqrt{c_2}\lambda_{d}\big(\mathbb{B}_d(R)\big)^2 + 4\sqrt[4]{c_2 c_3}c_{d,R,T}\lambda_{d}\big(\mathbb{B}_d(R)\big) + \sqrt{c_3}c_{d,R,T}^2\big).
	\end{align*}
	Next we take care of $\alpha_2(G)$. We use the Cauchy-Schwarz inequality to get
	\begin{align*}
		\alpha_2(G) &= \int_{(\RR^d \times \MM)^3} \EE[(\mathrm{D}_{x,z}^2G)^2 (\mathrm{D}_{y,z}^2G)^2] ~\Lambda^3\big(\mathrm{d}(x,y,z)\big)\\
		&\leq \int_{(\overline{W}_{R,s} \times \MM)^3} \sqrt{\EE[\big(\mathrm{D}_{x,z}^2G\big)^4]}\sqrt{\EE[\big(\mathrm{D}_{y,z}^2G\big)^4]}~\Lambda^3 \big(\mathrm{d}(x,y,z)\big)
	\end{align*}
	and apply the same decomposition argument as in the treatment of $\alpha_1(G)$.
	\begin{align*}
		\alpha_2(G) \leq& \lambda_{d}(\overline{W}_{R,s}) \lambda_{d}\big(\mathbb{B}_d(R)\big)^2 \frac{2^4 \gamma^3 \cdot c_2}{c_1^2 \cdot \lambda_{d+1}(W_s)^2}\\
		&+ 2\cdot\lambda_{d}\big(\mathbb{B}_d(R)\big)\lambda_{d}(\overline{W}_{R,s})^2 \frac{4\gamma^3\sqrt{c_2}}{c_1 \cdot \lambda_{d+1}(W_s)}\frac{\sqrt{c_3}}{c_1 \cdot \lambda_{d+1}(W_s)^3}\\
		&+ \lambda_{d}(\overline{W}_{R,s})^3 \frac{\gamma^3c_3}{c_1^2 \cdot \lambda_{d+1}(W_s)^6}\\
		\leq& \frac{\gamma^3c_{d,R,T}}{c_1^2 \cdot \lambda_{d+1}(W_s)}\big(16c_2\lambda_{d}\big(\mathbb{B}_d(R)\big)^2 + 8\sqrt{c_2c_3}c_{d,R,T}\lambda_{d}\big(\mathbb{B}_d(R)\big) + c_3c_{d,R,T}^2 \big).
	\end{align*}
	Finally, we use \eqref{Equation_Difference_Operator_of_Standardization_of_G_Linear}, (B) and then note that (C) implies $\EE[|\mathrm{D}_{x}f(Z(\xi \cap W_s))|^3] \leq 1 + c_2$ to get
	\begin{align*}
		\alpha_3(G) &\leq \int_{\overline{W}_{R,s} \times \MM} \EE[|\mathrm{D}_{x}G|^3] ~\Lambda\big(\mathrm{d}x\big)\\
		&= \int_{\overline{W}_{R,s} \times \MM}\EE\left[\frac{|\mathrm{D}_{x}f(Z(\xi \cap W_s))|^3}{\sqrt{\VV[f(Z(\xi \cap W_s))]}^3}\right] ~\Lambda\big(\mathrm{d}x\big)\\
		&\leq \int_{\overline{W}_{R,s} \times \MM} \frac{1}{c_1^\frac{3}{2} \lambda_{d+1}(W_s)^\frac{3}{2}} \cdot \EE[|\mathrm{D}_{x}f(Z(\xi \cap W_s))|^3] ~\Lambda\big(\mathrm{d}x\big)\\
		&\leq \frac{\gamma(1+c_2)\cdot c_{d,R,T}}{c_1^\frac{3}{2} \sqrt{\lambda_{d+1}(W_s)}}.
	\end{align*}
	Using these bounds, Theorem \ref{Theorem_LP18_Limit_Theorem} now gives us the asserted statement.
\end{proof}

\section{Central Limit Theorems}\label{Section_Main_results_and_proofs_Cylinders}

We will now present our limit theorems on the asymptotic behaviour of $Z(\xi)$. We start with the volume of the union of time bounded cylinders restricted to a compact window of observation. To be more precise, we are interested in the asymptotic distribution of the random variable
\begin{equation*}
	\lambda_{d+1}(Z(\xi) \cap W_s)
\end{equation*}
where, as before, $W_s := [-\frac{s}{2}, \frac{s}{2}]^{d} \times [0,T]$. Additionally to Theorem \ref{Theorem_Intro_TBC_Vol} we will provide a rate of convergence and in fact prove the following:

\begin{theorem}\label{Theorem_TBC_volume_limit_theorem}
	Let $N$ denote a standard normal random variable and $Z(\xi)$ as in Definition \ref{Definition_TBC_model}. Assume $r<s$. Then there exists a constant $c \in \RR_+$ such that
	\begin{equation*}
		\mathrm{d}_W\Big(\frac{\lambda_{d+1}(Z(\xi) \cap W_s) - \EE[\lambda_{d+1}(Z(\xi) \cap W_s)]}{\sqrt{\VV[\lambda_{d+1}(Z(\xi) \cap W_s)]}}, N\Big) \leq \frac{c}{\sqrt{\lambda_{d+1}(W_s)}}.
	\end{equation*}
\end{theorem}

We will remark that the capacity functional of the model admits a beautiful formula for the expected volume. An application of Fubini's theorem and \eqref{Equation_Probability_x_hits_TBC} immediately yield
\begin{align*}\label{expectedvolume}
	\EE[\lambda_{d+1}(Z(\xi) \cap W_s)] &= \EE\Big[\int_{W_s} \mathbbm{1}(x \in Z(\xi)) ~\lambda_{d+1}(\mathrm{d}x)\Big]\nonumber\\
	&= \big(1 - \exp(-\gamma \cdot r^d \kappa_d)\big) \cdot \lambda_{d+1}(W_s).
\end{align*}

Our second main result concerns itself with isolated cylinders. 
These represent nodes that have no contact to other nodes for the time frame at all. As before, we prove Theorem \ref{Theorem_Intro_TBC_Iso} by proving a stronger result offering rates of convergence. We start by considering the point process of isolated cylinders $\xi_{\text{Iso}}$ where
\begin{equation*}
	\xi_{\text{Iso}}(A) = \int_{\RR^d \times \MM} \mathbbm{1} \big(p \in A\big) \cdot \mathbbm{1}\big(\mathrm{Cyl}(p,v) \cap Z(\xi - \delta_{(p,v)}) = \emptyset \big) ~\xi \big(\mathrm{d}(p,v)\big)
\end{equation*}
counts the isolated cylinders with basepoint in $A \subset \RR^d$. Let us make a few remarks about its expectation. By the Mecke equation and Lemma \ref{Lemma_TBC_Capacity Functional} we have
\begin{align*}
	\EE[\xi_{\text{Iso}}(A)] &= \int_{A \times \MM} \PP[\mathrm{Cyl}(x) \cap Z(\xi) = \emptyset] ~\Lambda\big(\mathrm{d}x\big) \\
	&= \int_{A \times \MM} 1-\exp\Big(-\gamma \cdot \int_{\MM} \lambda_d\big(\overline{\mathrm{Cyl}(x)}_w + \mathbb{B}_d(r)\Big) ~\QQ(\mathrm{d}w)\big)~\Lambda\big(\mathrm{d}x\big) \\
	&= \gamma \cdot \int_{\MM} 1 - \exp\Big(-\gamma \cdot \int_{\MM} \lambda_d\big(\overline{\mathrm{Cyl}(0,v)}_w + \mathbb{B}_d(r)\big) ~\QQ(\mathrm{d}w)\Big) \QQ(\mathrm{d}v) \cdot \lambda_d(A)
\end{align*}
for the intensity of the point process $\xi_{\text{Iso}}$. Taking \eqref{Equation_Probability_Cylinder_hits_TBC} into account yields
\begin{equation*}
	\gamma \lambda_d(A) \cdot \exp(-\gamma (R_{h}+r)2^{d+1}r^d\kappa_d) \leq \EE[\xi_{\text{Iso}}(A)]  \leq \gamma \lambda_d(A) \cdot \exp(-\gamma (2r)^d \kappa_d).
\end{equation*}
As before, let $s>0$ and $W_s = [-\frac{s}{2},\frac{s}{2}]^d \times [0,T]$. By
\begin{equation*}
	\mathrm{Iso}_{Z(\xi)}(W_s) := \xi_{\text{Iso}}(W_s)
\end{equation*}
we define the random variable counting the number of isolated cylinders with basepoint in $W_s$. We present the following limit theorem.

\begin{theorem}\label{Theorem_TBC_isolated_nodes_limit_theorem}
	Assume $s \geq 6(R_h+r)$ and let $N$ denote a standard normal random variable. There exists a constant $c \in \RR_+$ such that
	\begin{equation*}
		\mathrm{d}_W\Big(\frac{\mathrm{Iso}_{Z(\xi)}(W_s) - \EE[\mathrm{Iso}_{Z(\xi)}(W_s)]}{\sqrt{\VV[\mathrm{Iso}_{Z(\xi)}(W_s)]}}, N\Big) \leq \frac{c}{\sqrt{\lambda_{d+1}(W_s)}}.
	\end{equation*}
\end{theorem}
The last main result takes a first step in analyzing the topological structure of the TBC model and gives insight into its \textit{Euler characteristic}. By $\mathcal{R}^{d}$ let us denote the convex ring over $\RR^{d}$ that is the system containing all unions of finitely many compact, convex subsets of $\RR^{d}$ (cf. \cite[p.601]{Schneider2008}).

\begin{definition}\label{Defintion_Euler_Characteristic}
	In the context of convex geometry the \textbf{Euler characteristic} is a function
	\begin{equation*}
		\chi: \mathcal{R}^{d} \to \ZZ
	\end{equation*}
	with the properties
	\begin{itemize}
		\item $\chi(K)=1$ if $K \subset \RR^d$ convex,
		\item $\chi(\emptyset)=0$ and
		\item $\chi(A \cup B) = \chi(A) + \chi(B) - \chi(A \cap B)$ for all $A,B \subset \RR^d$.
	\end{itemize}
\end{definition}
For reference, see \cite[p. 625]{Schneider2008}. Since our cylinders clearly are convex sets in $\RR^{d+1}$ and $\eta$ is a Poisson process with finite intensity, $\chi(Z(\xi) \cap W_s)$ is well defined for $W_s$ as given above. Theorem \ref{Theorem_Intro_TBC_Euler} follows from:

\begin{theorem}\label{Theorem_TBC_Euler_characteristic_limit_theorem}
	Assume $s \geq 6(R_h+r)$ and let $N$ denote a standard normal random variable. There exists a constant $c \in \RR_+$ such that
	\begin{equation*}
		\mathrm{d}_W\Big(\frac{\chi(Z(\xi) \cap W_s) - \EE[\chi(Z(\xi) \cap W_s)]}{\sqrt{\VV[\chi(Z(\xi) \cap W_s)]}}, N\Big) \leq \frac{c}{\sqrt{\lambda_{d+1}(W_s)}}.
	\end{equation*}
\end{theorem}

\subsection{Proof for the Covered Volume}
We start with the proof of the limit theorem concerning the covered volume of $Z(\xi)$ in $W_s$.

\begin{proof}[Proof of Theorem \ref{Theorem_TBC_volume_limit_theorem}]
	We aim to use Lemma \ref{Lemma_Apply_Malliavin_Stein_Bound} with
	\begin{equation*}
		f = \lambda_{d+1}.
	\end{equation*}
	Clearly $\lambda_{d+1}(A)<\infty$ for all $A \in \mathcal{C}^{d+1}$. Next, note that for $x = (p,v) \in \RR^d \times \MM$ we $\PP \text{-a.s.}$ have $\mathrm{D}_x\lambda_{d+1}(Z(\xi) \cap W_s) = 0$ for $|p| > R_h + s$ since in this case $\mathrm{Cyl}(x) \cap W_s = \emptyset$. If we set further $y = (q,w) \in \RR^d \times \MM$ we have $\mathrm{D}_{x,y}^2\lambda_{d+1}(Z(\xi) \cap W_s) = 0$ for $|p - q| > 2(R_h+r)$ $\PP \text{-a.s.}$ since $\mathrm{Cyl}(x) \cap \mathrm{Cyl}(y) = \emptyset$:
	\begin{align*}
		\mathrm{D}_{x,y}^2\lambda_{d+1}(Z(\xi) \cap W_s) &= \mathrm{D}_x \mathrm{D}_y ~\lambda_{d+1}(Z(\xi)\cap W_s)\\
		&= \mathrm{D}_x \bigl( \lambda_{d+1}(Z(\xi+\delta_y)\cap W_s) - \lambda_{d+1}(Z(\xi)\cap W_s)\bigr)\\
		&= \mathrm{D}_x \lambda_{d+1}(Z(\xi+\delta_y)\cap W_s) - \mathrm{D}_x \lambda_{d+1}(Z(\xi)\cap W_s) = 0
	\end{align*}
	$\PP \text{-a.s.}$. Thus conditions (A) and (D) are fulfilled and we have $c_3 = 0$. Condition (C) is next. The maximum volume of any cylinder is that of one with maximal scope, see Definition \ref{Defininition_Cylinder_Scope}. Since the scope of a cylinder $\Cyl(p,v)$ with $v \in \MM$ is only dependent on the last entry of $v$,  without loss of generality we can choose $v_h= (\sqrt{1-h^2},0,...,0,h)$ as a cylinder with maximum scope and receive
	\begin{equation}\label{Equation_Volume_Maximum_Cylinder}
		\lambda_{d+1}\big(\mathrm{Cyl}(p,v)\big) \leq \lambda_{d+1}\big(\mathrm{Cyl}(0,v_h)\big) \leq \frac{T \cdot r^d \kappa_d}{h}
	\end{equation}
	for all $(p,v) \in \RR^d \times \MM$. When we add a cylinder to the model, in the extreme case it becomes a subset of $W_s$ and intersects with none of the preexisting cylinders. In that case it adds the entirety of its volume to $\lambda_{d+1}(Z(\xi) \cap W_s)$. We can conclude
	\begin{equation*}
		\mathrm{D}_{(p,v)} \lambda_{d+1}(Z(\xi) \cap W_s) \leq \lambda_{d+1}\big(\mathrm{Cyl}(p,v)\big) \quad \PP \text{-a.s.}
	\end{equation*}
	and thus from \eqref{Equation_Volume_Maximum_Cylinder} that
	\begin{equation*}
		\EE\big[|\mathrm{D}_{(p,v)} \lambda_{d+1}(Z(\xi) \cap W_s)|^4\big]  \leq \Big(\frac{T \cdot r^d \kappa_d}{h}\Big)^4 =: c_2
	\end{equation*}
	and thus get condition (C). It remains to prove the lower bound on the variance. An application of Fubini's theorem yields
	\begin{align*}
		&\VV[\lambda_{d+1}(Z(\xi) \cap W_s)] \nonumber \\&= \EE\Big[\int_{W_s^2} \mathbbm{1}(x \in Z(\xi), y \in Z(\xi))~\lambda_{d+1}^2(\mathrm{d}(x,y)) \Big]  - \EE\Big[ \int_{W_s} \mathbbm{1}(x \in Z(\xi)) ~\lambda_{d+1}(\mathrm{d}x) \Big]^2 \nonumber \\
		&= \int_{W_s^2} \PP(x \in Z(\xi), y\in Z(\xi)) - \PP(x \in Z(\xi))\PP(y \in Z(\xi)) ~\lambda_{d+1}^2(\mathrm{d}(x,y)) \nonumber \\
		&= \int_{W_s^2} \PP(x \notin Z(\xi), y \notin Z(\xi)) - \PP(x \notin Z(\xi))\PP(y \notin Z(\xi)) ~\lambda_{d+1}^2(\mathrm{d}(x,y)),
	\end{align*}
	where the last equality follows from an application of DeMorgan's laws. Lemma \ref{Lemma_TBC_Capacity Functional} now yields
	\begin{align*}
		\PP(x \notin Z(\xi), y\notin Z(\xi)) &= \PP\big(\{x,y\} \cap Z(\xi) = \emptyset\big)\\ &= \exp\big(-\gamma \cdot \int_{\MM} \lambda_d\big(\overline{\{x,y\}_v} + \mathbb{B}_d(r)\big) ~\QQ(\mathrm{d}v)\big).
	\end{align*}
	To estimate this probability, consider some $x,y \in W_s$ and assume $\|x-y\| \leq r$. Then, the distance of their respective $v$-shadows satisfies the same, i.e. $\|\overline{x}_v - \overline{y}_v\| \leq r$ for $\overline{x}_v := \overline{\{x\}}_v$ and $\overline{y}_v := \overline{\{y\}}_v$. Also denote by $H_{w}^d(z)$ the cap of a $d$-dimensional hypersphere with radius $z$ and height $w$. Then
	\begin{equation*}
		\lambda_d(\mathbb{B}_d(r)) \leq \lambda_{d}\Big(\big(\overline{\{x\}_v} + \mathbb{B}_d(r)\big) \cup \big(\overline{\{y\}_v} + \mathbb{B}_d(r)\big)\Big) \leq 2\lambda_d(\mathbb{B}_d(r)) - 2\lambda_{d}\big(H_{\frac{r}{2}}^d(r)\big).
	\end{equation*}
	Since $e^x>1$ for all $x>0$ we get
	\begin{align*}
		&\PP(x \notin Z(\xi), y\notin Z(\xi)) - \PP(x \notin Z(\xi))\PP(y \notin Z(\xi))\\
		&\geq \exp\big(-2\gamma \lambda_d(\mathbb{B}_d(r))+2\gamma\cdot \lambda_d(H_{\frac{r}{2}}^d(r))\big) - \exp\big(-2 \gamma \lambda_d(\mathbb{B}_d(r))\big) =: \tau > 0
	\end{align*}
	for all $x \in \mathbb{B}_{d+1}(y,r)$ and $y \in W_s$. Then
	\begin{align*}
		&\VV[\lambda_{d+1}(Z(\xi) \cap W_s)]\nonumber \\
		&\geq \int_{W_s} \int_{\mathbb{B}_{d+1}(y,r) \cap W_s} \!\!\!\PP(x \notin Z(\xi), y \notin Z(\xi)) - \PP(x \notin Z(\xi))\PP(y \notin Z(\xi)) ~\lambda_{d+1}(\mathrm{d}x) \lambda_{d+1}(\mathrm{d}y) \nonumber \\
		&\geq \int_{W_s} \int_{\mathbb{B}_{d+1}(y,r) \cap W_s} \tau ~\lambda_{d+1}(\mathrm{d}x)~\lambda_{d+1}(\mathrm{d}y)\\
		&= \tau \int_{W_s} \lambda_{d+1}\Big(\mathbb{B}_{d+1}\left(y,r\right) \cap W_s\Big) ~\lambda_{d+1}(\mathrm{d}y).
	\end{align*}
	\noindent For $y \in W_s$ we also have that
	\begin{align*}
		\lambda_{d+1}\Big(\mathbb{B}_{d+1}\left(y,r\right) \cap W_s\Big) &\geq \frac{1}{2^{d+1}} \cdot \lambda_{d+1}(\mathbb{B}_{d+1}\big(0,\min\{r, T\}\big)) \\
		&= \frac{\kappa_{d+1} \cdot \min\{r, T\}^{d+1}}{2^{d+1}},
	\end{align*}
	where we have equality if $y$ sits just in a corner of the observation window $W_s$. For instance in the case $d=1$, $r<T$ and $y=\left(-\frac{s}{2},T\right)$ one quarter of $\mathbb{B}_{2}\left(y,r\right)$ intersects with $W_s$.  Now we can derive the bound
	\begin{align*}
		&\VV[\lambda_{d+1}(Z(\xi) \cap W_s)] \geq \tau \int_{W_s} \frac{\kappa_{d+1} \cdot \min\{r, T\}^{d+1}}{2^{d+1}} ~\lambda_{d+1}(\mathrm{d}y)\\
		&= \frac{\tau \cdot \kappa_{d+1} \cdot \min\{r, T\}^{d+1}}{2^{d+1}} \cdot \lambda_{d+1}(W_s) =: c_1 \cdot \lambda_{d+1}(W_s).
	\end{align*}
	This shows condition $\mathrm{(B)}$ of Lemma \ref{Lemma_Apply_Malliavin_Stein_Bound} and concludes the proof.
\end{proof}

Taking up Remark \ref{Remark_TBC_Upper_Bound_Variance}, we note that both lower and upper bound on the variance are of the same order.

\subsection{Proof for the Isolated Nodes}
Next we deal with the limit theorem for the isolated cylinders.

\begin{proof}[Proof of Theorem \ref{Theorem_TBC_isolated_nodes_limit_theorem}]
	As we did for the covered volume, we aim to apply Lemma \ref{Lemma_Apply_Malliavin_Stein_Bound} to the function
	\begin{equation*}
		\mathrm{Iso}_{Z(\xi)}: \xi \mapsto n \in \NN
	\end{equation*}
	counting the number of isolated cylinders created by $\xi$. Since the Poisson process of basepoints $\eta$ has finite intensity, it holds that $\mathrm{Iso}_{Z(\xi)}(A)< \infty$ $\PP$-a.s. for all $A \in \mathcal{C}^{d+1}$. Next we verify condition (A) and (D). Set $R = 2(R_h + r)$ and $c_3 =0$. Since two cylinders can only intersect if the distance between their basepoints is smaller than $R$, both conditions are satisfied. This argument follows analogously to our reasoning in the proof of Theorem \ref{Theorem_TBC_volume_limit_theorem}.
	For the same reason, we can bound the first order difference operator by
	\begin{equation*}
		|\mathrm{D}_{x}\mathrm{Iso}_{Z(\xi)}(W_s)| \leq \eta\big(\mathbb{B}_d(0,R)\big) \quad \PP \text{-a.s.}.
	\end{equation*}
	Set $a := \EE\big[\eta\big(\mathbb{B}_d(0,R)\big)\big] = \gamma \cdot \lambda_d\big(\mathbb{B}_ d(0,R)\big)$ and $x\in \RR^d \times \MM$. The number of isolated cylinders that can be connected by adding $\mathrm{Cyl}(x)$ to the model is limited by the total number of cylinders within maximum range. By using the fourth moment of the Poisson distributed random variable ${\eta\big(\mathbb{B}_d(0,R)\big)}$, we get
	\begin{equation}\label{iso4}
		\EE[\mathrm{D}_{x}\mathrm{Iso}_{Z(\xi)}(W_s)^4] \leq a^4 + 6a^3 + 7a^2 + a.
	\end{equation}
	Thus condition (C) is satisfied with $c_2 := a^4 + 6a^3 + 7a^2 + a + 1$. To verify (B), we use a strategy applied in \cite{Reitzner2005} and disassemble $\left[-\frac{s}{2}, \frac{s}{2}\right]^d$ into boxes of side length $6(R_h+r)$. Without loss of generality we can assume $s = n \cdot 6(R_h+r)$ for some $n \in \NN$, as the area outside these boxes is of no consequence to what follows. We index these sets by indices collected in the set $I$. For each $i \in I$ consider $C_i$, a box of side length $2(R_h + r)$ in the center of the corresponding greater box $S_i$ of side length $6(R_h + r)$. We define the event
	\begin{equation*}
		A_i = \{\eta(C_i)=2, ~\eta(S_i \backslash C_i)=0\}.
	\end{equation*}
	Let the $\sigma$-algebra $\mathcal{F}$ contain information on basepoints and directions of all cylinders of $Z(\xi)$, except those with basepoint in a square $S_i$ with $\1(A_i) = 1$. Now we decompose the variance
	\begin{align}\label{VarianceDecompIso}
		\VV[\mathrm{Iso}_{Z(\xi)}(W_s)] &= \EE[\VV[\mathrm{Iso}_{Z(\xi)}(W_s) ~|~ \mathcal{F}]] + \VV[\EE[\mathrm{Iso}_{Z(\xi)}(W_s) ~|~ \mathcal{F}]]\nonumber\\
		& \geq \EE[\VV[\mathrm{Iso}_{Z(\xi)}(W_s) ~|~ \mathcal{F}]]\nonumber\\
		& = \EE\left[\sum_{\1(A_i)=1} \VV_{X_{i,1},X_{i,2}}[\mathrm{Iso}_{Z(\xi)}(W_s)] \right]
	\end{align}
	where the right hand side variance is taken with respect only to the two random elements $X_{i,1},X_{i,2} \in \RR^d \times \MM$, created by the event $A_i$. Since these two cylinders can not intersect with any other cylinders of the model, the number of isolated nodes outside $A_i$ is independent of that contributed by $X_{i,1}$ and $X_{i,2}$. Let $\Delta_{X_{i,1},X_{i,2}}\mathrm{Iso}_{Z(\xi)}(W_s)$ denote that contribution. Note that
	\begin{equation}\label{Equation_Isolated_Cylinders_X1_X2_Contribution_Variance}
		\Delta_{X_{i,1},X_{i,2}}\mathrm{Iso}_{Z(\xi)}(W_s) = 2\cdot \1(\Cyl(X_{i,1}) \cap \Cyl(X_{i,2}) = \emptyset).
	\end{equation}
	This yields
	\begin{align*}
		\lefteqn{\VV_{X_{i,1},X_{i,2}}[\mathrm{Iso}_{Z(\xi)}(W_s)] }\\
		&= \VV_{X_{i,1},X_{i,2}}[\mathrm{Iso}_{Z(\xi - \delta_{X_{i,1}} - \delta_{X_{i,2}})}(W_s) + \Delta_{X_{i,1},X_{i,2}}\mathrm{Iso}_{Z(\xi)}(W_s)] \\
		&= \VV_{X_{i,1},X_{i,2}}[\Delta_{X_{i,1},X_{i,2}}\mathrm{Iso}_{Z(\xi)}(W_s)]\\
		&= 4 \cdot \PP[\Cyl(X_{i,1}) \cap \Cyl(X_{i,2}) = \emptyset] \cdot \PP[\Cyl(X_{i,1}) \cap \Cyl(X_{i,2}) \neq \emptyset] 
		=: c_4,
	\end{align*}
	since both probabilities clearly are strictly positive. Using this in \eqref{VarianceDecompIso} together with the stationarity and independence properties of $\eta$ as well as $|I| = \left(\frac{s}{6(R_h+r)}\right)^d$ we have
	\begin{align*}
		\VV[\mathrm{Iso}_{Z(\xi)}(W_s)] &\geq c_4 \cdot \sum_{i \in I} \PP[A_i] \\
		&\geq c_4 \cdot \sum_{i \in I} \PP[\eta(C_i)=2] \cdot \PP[\eta(S_i \backslash C_i)=0]\\
		&\geq c_4 \cdot \sum_{i \in I} c_5 = c_6 \cdot \left(\frac{s}{6(R_h+r)}\right)^d \\
		&= \frac{c_6}{6^d T(R_h+r)^d} \lambda_{d+1}(W_s)
	\end{align*}
	and thus condition (B). In the last equality we have used that $\lambda_{d+1}(W_s) = s^d \cdot T$.
\end{proof}

Again, the lower bound matches the order of the upper bound as given in Remark \ref{Remark_TBC_Upper_Bound_Variance}.

\subsection{Proof for the Euler Characteristic}

It remains to prove the limit theorem for the Euler characteristic.

\begin{proof}[Proof of Theorem \ref{Theorem_TBC_Euler_characteristic_limit_theorem}]
	We aim to use Lemma \ref{Lemma_Apply_Malliavin_Stein_Bound} on $f_{(\cdot)}$ for $f_{\xi}: \mathcal{C}^{d+1} \to \ZZ$ with
	\[
	f_{\xi}(A)=\left\{\begin{array}{ll} \chi(Z(\xi) \cap A), & \text{if $A  \in \mathcal{R}^{d+1}$}  \\
		0, & \text{else.}\end{array}\right.
	\]
	Since $\eta$ has finite intensity, the properties of $\chi$ imply that $f(A) < \infty$ for all $A \in \mathcal{C}^{d+1}$ $\PP$-a.s.. Once more we start by checking conditions (A) and (D). Again, set $R = 2(R_h+r)$ and $c_3 = 0$. As $\mathrm{Cyl}(p,v)\cap W_s = \emptyset$ for any $(p,v) \in \RR^d \times \MM$ with $|p| \geq R+s$, condition (A) is satisfied.\\
	 Also, for $(q,w)\in \RR^d \times \MM$, $\|p-q\| \geq R$ implies $\Cyl(p,v)\cap \Cyl(q,w) = \emptyset$ and thus $\mathrm{D}_{(p,v)}(\chi(Z(\xi+\delta_{(q,w)})\cap W_s)) - \mathrm{D}_{(p,v)}(\chi(Z(\xi)\cap W_s)) =0$ $\PP \text{-a.s.}$.
	To verify (C) we note that for $x\in \RR^d \times \MM$ the definition of $\chi$ gives us
	\begin{align}\label{DiffOpEuler}
		\mathrm{D}_{x} \chi(Z(\xi) \cap W_s) &= \chi\big(Z(\xi+\delta_x) \cap W_s\big) - \chi(Z(\xi) \cap W_s)\\
		&= \chi\big(\left(Z(\xi) \cup \Cyl(x)\right) \cap W_s\big) - \chi(Z(\xi) \cap W_s)\nonumber\\
		&= \chi\big((Z(\xi) \cap W_s) \cup (\Cyl(x)\cap W_s)\big) - \chi(Z(\xi) \cap W_s)\nonumber\\
		&= \chi\big(\mathrm{Cyl}(x) \cap W_s\big) - \chi\big(Z(\xi) \cap \mathrm{Cyl}(x) \cap W_s \big) \quad \PP \text{-a.s.}.\nonumber
	\end{align}
	Furthermore, it holds true that $\chi\big(\mathrm{Cyl}(x) \cap W_s\big) = \1(\mathrm{Cyl}(x) \cap W_s \neq \emptyset)$ and since $\chi\big(Z(\xi) \cap \mathrm{Cyl}(x) \cap W_s \big)$ counts the number of disjoint intersections of $\mathrm{Cyl}(x) \cap W_s$ with cylinders of $Z(\xi)$ we also have
	\begin{align*}
		\chi\big(Z(\xi) \cap \mathrm{Cyl}(x) \cap W_s \big) &\leq \sum_{y \in \xi} \mathbbm{1}\big(\mathrm{Cyl}(y) \cap \mathrm{Cyl}(x) \neq \emptyset\big)\leq \eta\big(\mathbb{B}_d(0,R)\big) \quad \PP \text{-a.s.}.
	\end{align*}
	To this, we have to add the number of cylinders added by the difference operators emerging in condition (C). Thus the condition is fulfilled with
	\begin{equation*}
		c_2 := \EE[(2+\eta\big(\mathbb{B}_d(0,R)\big))^4].
	\end{equation*}
	To verify (B) we use the same strategy as in the proof of Theorem \ref{Theorem_TBC_isolated_nodes_limit_theorem} and now have to consider $\VV_{X_{i,1},X_{i,2}}[\chi(Z(\xi) \cap W_s)]$. Again, since the cylinders created by $X_{i,1}$ and $X_{i,2}$ can not intersect, we can use the additivity of the Euler characteristic to get
	\begin{align*}
		&\VV_{X_{i,1},X_{i,2}}[\chi(Z(\xi) \cap W_s)]\\
		&\quad= \VV_{X_{i,1},X_{i,2}}[\chi(Z(\xi - \delta_{X_{i,1}} - \delta_{X_{i,2}}) \cap W_s) + \Delta_{X_{i,1},X_{i,2}}\chi(Z(\xi) \cap W_s)]\\
		&\quad= \VV_{X_{i,1},X_{i,2}}[\Delta_{X_{i,1},X_{i,2}}\chi(Z(\xi) \cap W_s)].
	\end{align*}
	Since $\Delta_{X_{i,1},X_{i,2}}\chi(Z(\xi) \cap W_s) \in \{1,2\}$ we get
	\begin{align*}
		&\VV_{X_{i,1},X_{i,2}}[\chi(Z(\xi) \cap W_s)]\\
		&\quad =5 \cdot \PP[\Cyl(X_{i,1}) \cap \Cyl(X_{i,2}) = \emptyset] \cdot \PP[\Cyl(X_{i,1}) \cap \Cyl(X_{i,2}) \neq \emptyset].
	\end{align*}
	Thus condition (B) follows as it did for Theorem \ref{Theorem_TBC_isolated_nodes_limit_theorem}.
\end{proof}

Note that, again, the order of the lower bound on the variance coincides with that of the upper bound found in Remark \ref{Remark_TBC_Upper_Bound_Variance}.

\section{Stacked Time Bounded Cylinders}\label{Section_TBC_Stacked_Cylinders}

\begin{figure}
	\begin{center}
		\includegraphics[scale=0.5]{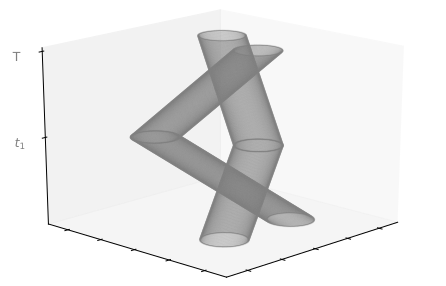}
	\end{center}
	\caption{An excerpt of the TBC model in the expanded setting.}
\end{figure}

Up to this point our model supports movements in one direction only. To allow for the nodes to change direction we iterate the marking process, giving a node a new direction vector at fixed points in time. These times can be given deterministically or randomly, but independent of all other random variables. In view of application, one could think of a deterministic setting like a daily routine or an exponentially distributed random variable for the time distances.

\subsection{Construction of the Cylinderstacks}

Set $t_0 := 0$, $t_{K+1} :=T$ and let $T_{K} = (t_0,t_1,t_2,\dots,t_{K+1})$, for some $K\in\NN$, be a strictly monotone sequence of times in $[0,T]$.
Now recall the definition of a cylinder as given in Definition \ref{Definition_Cylinder}. We will now modify it to reflect the expanded setting. To a point $p \in \eta$ let $V = (v_{0},\dots,v_{K})$ be the vector of directions taken, that is $v_{i} \in \MM$ for $i \in \{0,\dots,K\}$. Consider the projection $\pi: \MM \to [0,1]$ onto the velocity, that is $\pi(v) = v_{d+1}$ for $v\in \MM$, and recall that $\hat{p}_{0} \in \RR^{d+1}$ is the vector $p$ but with an additional coordinate with value zero. Then
\begin{equation*}
	\hat{p}_{k} := \hat{p}_{0} + \sum_{j=0}^{k-1} \frac{|t_{j+1} - t_j|}{\pi(v_{j})} v_{j}.
\end{equation*}
In the context of random networks, think of this point as the location at time $t_k$ of a node starting its journey in $p$. We are now in a position to modify Definition \ref{Definition_Cylinder} for this expanded setting.
\begin{definition}\label{Definition_Cylinderstack}
	Given a radius $r \geq 0$ and $T_K$ as above, the values of the function
	\begin{align*}
		\mathrm{Cyl}_{r,T_K}:~\RR^{d+1} \times (\MM)^K &\to \mathcal{C}^{d+1}\\
		(p,V) &\mapsto \bigcup_{k=0}^K \bigcup_{u \in [0,1]} \left( \Big(\hat{p}_k + \frac{u |t_{k+1}-t_{k}|}{\pi(v_{k})} \cdot v_{k}\Big) + \mathbb{B}_d(r)\right)
	\end{align*}
	are the \textbf{time bounded cylinderstacks}.
\end{definition}

Our idea is that at every time $t \in T_K$ a cylinder has the opportunity to change direction and will do so with probability $q \in [0,1]$. A new direction is always drawn with respect to the probability measure $\QQ$ introduced in section \ref{Section_Construction_of_Cylinders}. Thus the aforementioned vector $V = (v_{0},\dots,v_{K})\in (\MM)^K$ is drawn with respect to the probability measure
\begin{equation*}
	\QQ_K := \QQ \otimes (1-q)\delta_{v_0} + q\QQ \otimes \dots \otimes(1-q)\delta_{v_{K-1}} + q\QQ.
\end{equation*}
We can now give analogues to Definition \ref{Definition_Process_of_Cylinders} and Lemma \ref{Lemma_Xi_is_a_Poisson_Point_Process}.
\begin{definition}\label{Definition_Process_of_Cylinderstacks}
	Let $\eta$ be defined as in Section \ref{Section_Construction_of_Cylinders} and $\QQ_K$ as above. Mark every point in the support of $\eta$ with a direction from $\MM$, randomly with respect to $\QQ_K$. The resulting marked point process $\mathbf{\xi}_K$ defined on $\RR^d \times (\MM)^K$ is called the \textbf{process of time bounded cylinderstacks}.
\end{definition}

\begin{lemma}\label{Lemma_Xi_K_is_a_Poisson_Point_Process}
	$\xi_K$ is a Poisson point process with intensity measure $\Lambda_{K} = \gamma\lambda_d \otimes \QQ_K$.
\end{lemma}

\begin{proof}
	Since $\QQ_K$ is independent of the Poisson process $\eta$ we have an independent marking again and thus the statement follows as it did for Lemma \ref{Lemma_Xi_is_a_Poisson_Point_Process}.
\end{proof}

That the process of time bounded cylinderstacks preserves the Poisson properties of the original point process seems remarkable, but can be intuitively explained by using thinnings. Since for $t_1\geq T$ the model remains unchanged let us consider $t_1<T$. We start with the point process $\xi$ as introduced in Definition \ref{Definition_Process_of_Cylinders}. From Lemma \ref{Lemma_Xi_is_a_Poisson_Point_Process} we know that $\xi$ is a Poisson point process. Now we take a $q$-thinning of $\xi$, namely a Bernoulli experiment with parameter $q$ which decides for each point of the Poisson process if it belongs to $\xi(q)$ or $\xi-\xi(q)$. The Thinning Theorem \cite[Corollary 5.9]{Last2018} implies that both $\xi(q)$ and $ \xi-\xi(q)$ are Poisson processes.\\ For the points in $\xi(q)$ we now choose a new direction and by adding this into another component of the marking, we get a new marked point process which we call $\xi(q)^M$. Analogously to the single direction case the Marking Theorem implies that this process is Poisson. For points in $\xi-\xi(q)$ the direction stays the same, nevertheless, as for $\xi(q)$ we add the previous direction as a new marking and again obtain a marked point process $\xi(1-q)^M$. As a superposition of two Poisson processes $\xi_1:=\xi(q)^M \oplus \xi(1-q)^M$ is also Poisson.
We proceed in thinning, marking and gluing together for each time step $t_i<T$ and obtain the Poisson process of cylinderstacks. We can now conveniently define the expanded model as follows.

\begin{definition}\label{Definition_Stacked_TBC_model}
	With the point process $\xi_K$ given as above, the \textbf{stacked time bounded cylinder (sTBC) model} is the random set
	\begin{equation*}
		Z(\xi_K) := \bigcup_{(p,V) \in \xi_K} \mathrm{Cyl}(p,V).
	\end{equation*}
	In this notation we identify the random measure $\xi_K$ with its support.
\end{definition}

With these definitions we can now rephrase Lemma \ref{Lemma_Apply_Malliavin_Stein_Bound} to work with our modifications.

\begin{lemma}\label{Lemma_Apply_Malliavin_Stein_Bound_sTBC}
	Let $f: \mathcal{C}^{d+1} \to \RR$, $s\geq1$ and $W_s = [-\frac{s}{2}, \frac{s}{2}]^{d} \times [0,T]$. Assume that $|f(A)|<\infty$ for all $A \in \mathcal{C}^{d+1}$ and that there exist constants $c_1$, $c_2$, $c_3$ and $R \in \RR$ such that the following conditions are met for all $x = (p,V), y= (q,W) \in  \RR^d \times (\MM)^K$:
	\begin{itemize}
		\item[$\mathrm{(A)}$] $\mathrm{D}_xf(Z(\xi_K) \cap W_s) = 0 \quad \PP \text{-a.s.}$ \text{for} $\|p\| > R+s$,
		\item[$\mathrm{(B)}$] $\VV\big[f(Z(\xi_K) \cap W_s)\big] \geq c_1 \cdot \lambda_{d+1}(W_s)$,\vspace*{0.1cm}
		\item[$\mathrm{(C)}$] $\max\{\EE\big[\mathrm{D}_{x}f(Z(\xi_K) \cap W_s)^4\big],~\EE\big[\mathrm{D}_{x}f(Z(\xi_K+\delta_y) \cap W_s)^4\big]\} \leq c_2$,\vspace*{0.1cm}
		\item[$\mathrm{(D)}$] $\|p - q\|>R \Rightarrow \EE[\big(\mathrm{D}_{x,y}^2 f(Z(\xi_K) \cap W_s)\big)^4] \leq \frac{c_3}{\lambda_{d+1}(W_s)^4}$.
	\end{itemize}\vspace*{0.25cm}
	Then, there is a constant $c \in \RR$ such that, for a standard normal random variable $N$,
	\begin{equation*}
		\mathrm{d_W}\Bigl(\frac{f(Z(\xi_K) \cap W_s)-\EE\big[f(Z(\xi_K) \cap W_s)\big]}{\sqrt{\VV\big[f(Z(\xi_K) \cap W_s)\big]}}, N\Bigr) \leq \frac{c}{\sqrt{\lambda_{d+1}(W_s)}}.
	\end{equation*}
\end{lemma}

As the proof in section \ref{Section_Applying_the_Malliavin_Stein_Bound} relies only on the assumed bounds and not on how the model is constructed, we can use the same techniques and estimates we used there and prove this lemma analogously.

\subsection{Covered Volume}

As usual, let $W_s := [-\frac{s}{2}, \frac{s}{2}]^{d} \times [0,T]$. We will now study the covered volume of the sTBC model that is the random variable
\begin{equation*}
	\lambda_{d+1}(Z(\xi_K) \cap W_s).
\end{equation*}

As this is an additive function our strategy for working with this quantity will be to split the sTBC model into TBC models and evaluate the volume for each of them. To that end let $Z_t(\xi)$ denote the TBC model with time horizon $t \geq 0$. We will first apply our strategy when computing the expectation of the volume. Using \eqref{Equation_Probability_x_hits_TBC} and Fubini we have
\begin{align*}
	\EE[\lambda_{d+1}(Z(\xi_K) \cap W_s)] &= \EE\Big[\int_{W_s} \1(x \in Z(\xi_K)) ~\mathrm{d}x\Big]\\
	&= \EE\Big[\sum_{k=0}^{K} \int_{[-\frac{s}{2},\frac{s}{2}]^d \times [t_k, t_{k+1}]}\1(x \in Z_{|t_{k+1}-t_k|}(\xi))~\mathrm{d}x\Big]\\
	&= \EE\Big[\sum_{k=0}^{K} \EE\Big[\int_{[-\frac{s}{2},\frac{s}{2}]^d \times [t_k, t_{k+1}]}\1(x \in Z_{|t_{k+1}-t_k|}(\xi))~\mathrm{d}x ~|~ T_K\Big]\Big]\\
	&= (1 - \exp(-\gamma \cdot r^d \kappa_d)) \cdot \EE\left[\sum_{k=0}^{K} \lambda_{d+1}\left(\left[-\frac{s}{2},\frac{s}{2}\right]^d \times [t_k, t_{k+1}]\right)\right]\\
	&= (1 - \exp(-\gamma \cdot r^d \kappa_d)) \cdot \lambda_{d+1}(W_s).
\end{align*}
We see that the expected volume remains unchanged, which is unsurprising given the homogeneous nature of our model.
\vspace*{0.5cm}
\begin{theorem}\label{Theorem_TBC_Volume_Stacked}
	Let $N$ denote a standard normal random variable and assume $r<s$. There exists a constant $c \in \RR_+$ such that
	\begin{equation*}
		\mathrm{d}_W\Big(\frac{\lambda_{d+1}(Z(\xi_K) \cap W_s) - \EE[\lambda_{d+1}(Z(\xi_K) \cap W_s)]}{\sqrt{\VV[\lambda_{d+1}(Z(\xi_K) \cap W_s)]}}, N\Big) \leq \frac{c}{\sqrt{\lambda_{d+1}(W_s)}}.
	\end{equation*}
\end{theorem}
\vspace*{0.5cm}
\begin{proof}
	We use Lemma \ref{Lemma_Apply_Malliavin_Stein_Bound_sTBC} on the function $f=\lambda_{d+1}$. We can verify conditions (A) and (D) by the same observation we made in the proof of Theorem \ref{Theorem_TBC_volume_limit_theorem}, again using $R = 2(R_{h} + r)$ and $c_3=0$. The same is true for condition (C) which once more is satisfied with $c_2 := \Big(\frac{T \cdot r^d \kappa_d}{h}\Big)^4$.
	To verify (B) we have to make some adjustments but start as before with
	\begin{align*}
		&\VV[\lambda_{d+1}(Z(\xi_K) \cup W_s)]\\
		&\qquad= \int_{W_s^2} \PP(x \notin Z(\xi_K), y \notin Z(\xi_K)) - \PP(x \notin Z(\xi_K))\PP(y \notin Z(\xi_K)) ~\lambda_{d+1}^2(\mathrm{d}(x,y)).
	\end{align*}
	Given a point $y \in W_s$, let $t_-^y, t_+^y \in T_K$ such that $t_-^y \leq y_{d+1} \leq t_+^y$, $t_-^y \neq t_+^y$ and there is no element of $T_K$ in the interval $(t_-^y,t_+^y)$. We have
	\begin{align*}
		&\VV[\lambda_{d+1}(Z(\xi_K) \cap W_s)] \nonumber \\&= \int_{W_s^2} \PP(x \notin Z(\xi_K), y \notin Z(\xi_K)) - \PP(x \notin Z(\xi_K))\PP(y \notin Z(\xi_K)) ~\lambda_{d+1}^2(\mathrm{d}(x,y)) \\
		&\geq \int_{W_s} \int_{[-\frac{s}{2}, \frac{s}{2}]^d \times [t_-^y, t_+^y]} \PP(x \notin Z_{|t_+^y - t_-^y|}(\xi), y \notin Z_{|t_+^y - t_-^y|}(\xi))\\&\hspace{4.2cm}- \PP(x \notin Z_{|t_+^y - t_-^y|}(\xi))\PP(y \notin  Z_{|t_+^y - t_-^y|}(\xi)) ~\lambda_{d+1}(\mathrm{d}x)~\lambda_{d+1}(\mathrm{d}y).
	\end{align*}
	Note that the bound of the inner integral now ensures that $x$ and $y$ share a time slot in the stacked cylinder model. Because of this we can argue as in the single direction case and derive the constant $\tau>0$ such that  $p_{xy} - p_x p_y \geq \tau$ for all $x \in \mathbb{B}_{d+1}(y,\frac{r}{2}) \cap \big([-\frac{s}{2}, \frac{s}{2}]^d \times [t_-^y, t_+^y]\big)$ and $y\in W_s$. As before we then have
	\begin{align*}
		&\lambda_{d+1}\big(\mathbb{B}_d\big(y,\frac{r}{2}\big) \cap \big(\left[-\frac{s}{2}, \frac{s}{2}\right]^d \times [t_-^y, t_+^y]\big)\big)\\
		&\qquad \geq \frac{1}{2^{d+1}} \cdot \lambda_{d+1}(\mathbb{B}_{d+1}\big(0,\min\{\frac{r}{2}, t_+^y - t_-^y\}\big)) \\
		&\qquad = \frac{\kappa_{d+1} \cdot \min\{\frac{r}{2}, t_+^y - t_-^y\}^{d+1}}{2^{d+1}}
	\end{align*}
	and thus
	\begin{align*}
		&\VV[\lambda_{d+1}(Z(\xi_K) \cap W_s)]\\
		&\qquad \geq \tau \cdot \int_{W_s} \EE\big[\lambda_{d+1}(\mathbb{B}_d(y,r) \cap \big(\left[-\frac{s}{2}, \frac{s}{2}\right]^d \times [t_-^y, t_+^y]\big))\big] ~\lambda_{d+1}(\mathrm{d}(y))  \\
		&\qquad \geq \frac{\tau \cdot \kappa_{d+1} \cdot \EE[\min\{r, Y\}^{d+1}]}{2^{d+1}} \cdot \lambda_{d+1}(W_s) =: c_1 \cdot \lambda_{d+1}(W_s)
	\end{align*}
	with the random variable $Y := \underset{\text{$k$}}{\min}\{|t_{k+1} - t_{k}|\}$.
\end{proof}

\subsection{Isolated Cylinders}

We are now interested in the isolated stacks in the sTBC model. The point process of isolated cylinder stacks with basepoint in $A \subset \RR^d$ is given by
\begin{equation*}
	\mu(A) = \int_{A \times \big(\MM \big)^K} \mathbbm{1}\big(\mathrm{Cyl}(x) \cap Z(\xi_K - \delta_{x}) = \emptyset \big) ~\xi_K \big(\mathrm{d}x\big).
\end{equation*}

The next theorem shows asymptotic normality of $\mathrm{Iso}_{Z(\xi_K)}(W_s) = \mu(W_s)$.

\begin{theorem}\label{Theorem_TBC_Iso_Stacked}
	Let $N$ denote a standard normal random variable and assume $s \geq 6(R_h+r)$. There exists a constant $c \in \RR_+$ such that
	\begin{equation*}
		\mathrm{d}_W\Big(\frac{\mathrm{Iso}_{Z(\xi_K)}(W_s) - \EE[\mathrm{Iso}_{Z(\xi_K)}(W_s)]}{\sqrt{\VV[\mathrm{Iso}_{Z(\xi_K)}(W_s)]}}, N\Big) \leq \frac{c}{\sqrt{\lambda_{d+1}(W_s)}}.
	\end{equation*}
\end{theorem}

\begin{proof}
	Note that the maximum scope of a cylinder remains unchanged in the multi-directional setting. Thus, conditions (A), (C) and (D) are derived in the same way as in the proof of Theorem \ref{Theorem_TBC_isolated_nodes_limit_theorem}, using the same choices for $R$, $c_2$ and $c_3$. To verify (B) we use the same strategy as before and get to
	\begin{align*}
		&\VV_{X_{i,1},X_{i,2}}[\mathrm{Iso}_{Z(\xi_K)}(W_s)] \\
		&\quad = 4 \cdot \PP[\Cyl(X_{i,1}) \cap \Cyl(X_{i,2}) = \emptyset] \cdot \PP[\Cyl(X_{i,1}) \cap \Cyl(X_{i,2}) \neq \emptyset].
	\end{align*}
	Even in this modified setting it is clear that these probabilities are not zero, as they can be estimated using $r$ and $R$. It follows that (B) holds true, which concludes the proof.
\end{proof}

\bibliographystyle{alpha}
\bibliography{Library}

\newcommand{\etalchar}[1]{$^{#1}$}
\begin{thebibliography}{DFHO{\etalchar{+}}18}

\bibitem[ABV{\etalchar{+}}12]{Anthony2012}
D.~Anthony, W.P. Bennett, M.C. Vuran, M.B. Dwyer, S.~Elbaum, A.~Lacy,
  M.~Engels, and W.~Wehtje.
\newblock Sensing through the continent: Towards monitoring migratory birds
  using cellular sensor networks.
\newblock {\em 2012 ACM/IEEE 11th International Conference on Information
  Processing in Sensor Networks (IPSN)}, pages 329--340, 2012.

\bibitem[BBGT20]{Baci2019}
A.~Baci, C.~Betken, A.~Gusakova, and C.~Th\"{a}le.
\newblock Concentration inequalities for functionals of {P}oisson cylinder
  processes.
\newblock {\em Electronic Journal of Probability}, 25:1--27, 2020.

\bibitem[BST21]{Betken2021}
C.~Betken, M.~Schulte, and C.~Thäle.
\newblock Variance asymptotics and central limit theory for geometric
  functionals of poisson cylinder processes.
\newblock Electronic preprint arXiv:2111.04608, 2021.

\bibitem[BT16]{Broman2016}
E.~Broman and J.~Tykesson.
\newblock Connectedness of {P}oisson cylinders in {E}uclidean space.
\newblock {\em Annals de l'Institut Henry Poincaré Probabilités et
  Statistiques}, 52:102--126, 2016.

\bibitem[CBD02]{Camp2002}
T.~Camp, J.~Boleng, and V.~Davies.
\newblock A survey of mobility models for ad hoc network research.
\newblock {\em Wireless communications and mobile computing}, 2(5):483--502,
  2002.

\bibitem[DFHO{\etalchar{+}}18]{Dede2018}
J.~Dede, A.~Förster, E.~Hernández-Orallo, J.~Herrera-Tapia, K.~Kuladinithi,
  V.~Kuppusamy, P.~Manzoni, A.~bin Muslim, A.~Udugama, and Z.~Vatandas.
\newblock Simulating opportunistic networks: Survey and future directions.
\newblock {\em IEEE Communications Surveys Tutorials}, 20(2):1547--1573, 2018.

\bibitem[DFK16]{DFK16}
H.~Döring, G.~Faraud, and W.~König.
\newblock Connection times in large ad-hoc mobile networks.
\newblock {\em Bernoulli}, 22(4):2143 -- 2176, 2016.

\bibitem[DMPG09]{DMPG09}
J.~Díaz, D.~Mitsche, and X.~Pérez-Giménez.
\newblock Large connectivity for dynamic random geometric graphs.
\newblock {\em IEEE Trans. Mob. Comput.}, 8:821--835, 06 2009.

\bibitem[FH21]{Flimmel2021}
D.~Flimmel and L.~Heinrich.
\newblock On the variance of the area of planar cylinder processes driven by
  {B}rillinger-mixing point processes.
\newblock Electronic preprint arXiv:2104.10224, 2021.

\bibitem[HJC21]{HJC21}
C.~Hirsch, B.~Jahnel, and E.~Cali.
\newblock Percolation and connection times in multi-scale dynamic networks.
\newblock Electronic preprint arXiv:2103.03171, 2021.

\bibitem[HS09]{Hein2009}
L.~Heinrich and M.~Spiess.
\newblock Berry–{E}sseen bounds and {C}ramér-type large deviations for the
  volume distribution of {P}oisson cylinder processes.
\newblock {\em Lithuanian Mathematical Journal}, 49:381--398, 2009.

\bibitem[HS13]{Hein2013}
L.~Heinrich and M.~Spiess.
\newblock Central limit theorems for volume and surface content of stationary
  {P}oisson cylinder processes in expanding domains.
\newblock {\em Advances in Applied Probability}, 45:312--331, 2013.

\bibitem[LP18]{Last2018}
G.~Last and M.~Penrose.
\newblock {\em Lectures on the {P}oisson process}, volume~7 of {\em Institute
  of Mathematical Statistics Textbooks}.
\newblock Cambridge University Press, Cambridge, 2018.

\bibitem[Rei05]{Reitzner2005}
M.~Reitzner.
\newblock Central limit theorems for random polytopes.
\newblock {\em Probability Theory and Related Fields}, 133:483–507, 2005.

\bibitem[RMSM01]{Royer2001}
E.M. Royer, P.M. Melliar-Smith, and L.E. Moser.
\newblock An analysis of the optimum node density for ad hoc mobile networks.
\newblock {\em ICC 2001. IEEE International Conference on Communications.
  Conference Record (Cat. No.01CH37240)}, 3:857--861, 2001.

\bibitem[SW08]{Schneider2008}
R.~Schneider and W.~Weil.
\newblock {\em Stochastic and integral geometry}.
\newblock Probability and its Applications (New York). Springer-Verlag, Berlin,
  2008.

\bibitem[Wei87]{Weil1987}
W.~Weil.
\newblock Point processes of cylinders, particles and flats.
\newblock {\em Acta Applicandae Mathematicae}, 9:103--136, 1987.

\end{thebibliography}
\end{document}